\documentclass{amsart}

\usepackage{amsthm}
\usepackage{amssymb}
\usepackage{amsmath}
\usepackage{amsfonts}
\usepackage{amsrefs}
\usepackage{textcomp}
\usepackage[T1]{fontenc}
\usepackage[utf8x]{inputenc}
\usepackage{textcomp}
\usepackage{wasysym}
\usepackage{stmaryrd}
\usepackage{esint}
\usepackage[all]{xy}
\usepackage{graphicx}
\usepackage{bbm}
\usepackage{color}
\usepackage{hyperref}

\newtheorem{theorem}{Theorem}[section]
\newtheorem{definition}[theorem]{Definition}
\newtheorem{remark}[theorem]{Remark}

\newtheorem{lemma}[theorem]{Lemma}
\newtheorem{proposition}[theorem]{Proposition}
\newtheorem{corollary}[theorem]{Corollary}

\numberwithin{equation}{section}

\def\XXint#1#2#3{{\setbox0=\hbox{$#1{#2#3}{\int}$}
     \vcenter{\hbox{$#2#3$}}\kern-.5\wd0}}

\newcommand{\twopartdef}[4]
{
\left\{
		\begin{array}{ll}
			#1 & #2 \\
			#3 & #4
		\end{array}
	\right.
}

\usepackage{geometry}
\author{Samer Dweik, Wojciech G\'{o}rny}

\address{S. Dweik: Laboratoire de Math\'ematiques de Versailles, UVSQ, Universit\'e Paris-Saclay, 45 avenue des Etats-Unis, F-78035 Versailles Cedex.}
\address{W. G\'{o}rny: Faculty of Mathematics, University of Vienna, Oskar-Morgernstern-Platz 1, 1090 Wien; Faculty of Mathematics, Informatics and Mechanics, University of Warsaw, Banacha 2, 02-097 Warsaw, Poland.}
\email{samer.dweik@uvsq.fr, wojciech.gorny@univie.ac.at}

\date{\today}

\subjclass[2020]{35J25, 35J75, 49Q22}

\title[Weighted least gradient problem via optimal transport]{Optimal transport approach to Sobolev regularity \\ of solutions to the weighted least gradient problem}

\keywords{Least Gradient Problem, Optimal transport, Anisotropy}

\begin{document}

\begin{abstract}
We study the equivalence between the weighted least gradient problem and the weighted Beckmann minimal flow problem or equivalently, the optimal transport problem with Riemannian cost. Thanks to this equivalence, we prove existence and uniqueness of a solution to the weighted least gradient problem. Then, we show $L^p$ regularity on the transport density between two singular measures in the corresponding equivalent Riemannian optimal transport formulation. This will imply $W^{1,p}$ regularity of the solution of the weighted least gradient problem. 
\end{abstract}

\maketitle

\section{Introduction}
 
In this paper, we study a variant of the least gradient problem. The classical version of the least gradient problem consists of minimizing the total variation of the vector measure $Du$ among all BV functions $u$ defined on an open bounded domain $\Omega$ with given
boundary datum $g$:
\begin{equation}\label{BV least gradient problem classical}
\inf\bigg\{\int_{\Omega} |D u|\,:\,u \in BV(\Omega),\, u|_{\partial\Omega}=g\bigg\},
\end{equation}
where $u|_{\partial\Omega}$ denotes the trace of $u$ on the boundary $\partial\Omega$. It was first considered in this form in \cite{Sternberg}, where the authors prove existence and uniqueness of a solution to \eqref{BV least gradient problem classical} in the case where $\Omega$ is strictly convex and the boundary datum $g$ is continuous. The convexity assumption on the domain is crucial for the analysis performed in \cite{Sternberg}, and indeed the authors provide examples of boundary data for which there exist no solutions if it is violated. In the case when $\Omega$ is not strictly convex, one needs to introduce admissibility conditions on the boundary data to obtain existence and uniqueness of solutions, see \cite{Sabra,DG2019}. On the other hand, on strictly convex domains the continuity assumption on the boundary data can be relaxed; for instance, in \cite{Gor2018CVPDE} the author proved that in two dimensions Problem \eqref{BV least gradient problem classical} admits a solution as soon as $g \in BV(\partial\Omega)$ (for generalisations to higher dimensions, see \cite{Gor2021IUMJ,Mor}). This result is sharp in the sense that a counter-example given in \cite{Spradlin} shows that if $g \notin BV(\partial\Omega)$, solutions to Problem \eqref{BV least gradient problem classical} might not exist. Moreover, we lose uniqueness of the solution to Problem \eqref{BV least gradient problem classical} if $g \notin C(\partial\Omega)$.  \\



 
 
In \cite{GRS2017NA,DS}, the authors proved that in 2D, the problem \eqref{BV least gradient problem classical} is closely related to the following minimal flow formulation (which is called the Beckmann problem; see \cite{Beckmann}) \begin{equation} 
 \label{BeckmannIntro}\inf\bigg\{\int_{\overline{\Omega}} |v|\,:\,v \in \mathcal{M}(\overline{\Omega}, \mathbb{R}^2),\,\mathrm{div} \, v=0\,\,\,\,\mbox{and}\,\,\,v \cdot n =f:= \partial_\tau g \,\,\mbox{on}\,\,\,\partial\Omega\bigg\},
\end{equation}
provided that the domain $\Omega$ is convex and $g \in BV(\partial\Omega)$. Here,  the finite signed measure $\partial_\tau g$ is the tangential derivative of $g$. Moreover,  $\mathcal{M}(\overline{\Omega}, \mathbb{R}^2)$ denotes the set of vector measures over $\overline{\Omega}$ while the divergence constraint is understood in the distributional sense. In fact, the infima in the two problems coincide, and if $\Omega$ is strictly convex there is a one-to-one correspondence between vector measures $Du$ which are derivatives of solutions to \eqref{BV least gradient problem classical} and vector measures $v$ which solve \eqref{BeckmannIntro}. This equivalence is formally given by $v:=R_{\frac{\pi}{2}} Du$  on $\overline{\Omega}$. If $\Omega$ is only convex, the one-to-one correspondence is still true provided that for every optimal flow $v$ of the Beckmann problem, its total variation $|v|$ gives zero mass to the boundary. As a consequence, regularity estimates for one problem can be transferred to the other one, in particular the $L^p$ summability of the optimal flow $v$ yields $W^{1,p}$ regularity of the solution $u$ to Problem \eqref{BV least gradient problem classical}. We will describe this equivalence in more detail in Section \ref{sec:equivalenceeuclidean}.

On the other hand, the Beckmann problem \eqref{BeckmannIntro} is equivalent to the optimal transport problem with Euclidean cost (which is called the Monge-Kantorovich problem; see \cite{Kantorovich,Monge}) with source and target measures located on $\partial\Omega$ (see \cite{San2015}): 
\begin{equation}\label{KantoIntro}\min\bigg\{\int_{\overline{\Omega} \times \overline{\Omega}} |x-y|\,\mathrm{d}\Lambda\,:\,\Lambda \in  \mathcal{M}^+(\overline{\Omega} \times \overline{\Omega}),\,(\Pi_x)_{\#}\Lambda=f^{+} \,\,\mbox{and}\,\,(\Pi_y)_{\#} \Lambda =f^{-}\bigg\},
\end{equation}
where $f^+$ and $f^-$ are the positive and negative parts of $f$ (we note that $f^+$ and $f^-$ have the same total mass since $f$ is the tangential derivative of $g$ on the closed set $\partial\Omega$, so $f(\partial\Omega)=0$).  Moreover, it is well known (see, for instance, \cite{San2015}) that the Monge-Kantorovich problem \eqref{KantoIntro} has a dual formulation, which is the following:
\begin{equation}\label{Kanto dualIntro} 
    \sup\bigg\{\int_{\overline{\Omega}} \psi\,\mathrm{d}(f^+ - f^-)\,:\,\psi \in  \mbox{Lip}_1(\overline{\Omega})\bigg\}.
\end{equation}
From \cite{San2015}, every optimal flow $v$ of Problem \eqref{BeckmannIntro} is of the form $v=-|v|\nabla \psi$, where $\psi$ is a Kantorovich potential (i.e. a maximizer for the dual problem \eqref{Kanto dualIntro}). In addition, one can show that this nonnegative measure $|v|$ is a {\it{transport density}} between $f^+$ and $f^-$, i.e. there is an optimal transport plan $\Lambda$ such that 
\begin{equation*} \label{transport density definition Intro}
|v|(A)=\int_{\overline{\Omega} \times \overline{\Omega}} \mathcal{H}^1([x,y] \cap A)\,\mathrm{d}\Lambda (x,y),\,\,\,\mbox{for all Borel set}\,\,A \subset \overline{\Omega}.
\end{equation*}
We see that $|v|(\partial\Omega)=0$ as soon as $\Omega$ is strictly convex and so, the equivalence between Problems \eqref{BV least gradient problem classical} \& \eqref{BeckmannIntro} gives existence of a solution $u$ for Problem \eqref{BV least gradient problem classical}.  Moreover, classical results imply that the transport density $|v|$ is unique (i.e., it does not depend on the choice of the optimal transport plan $\Lambda$) provided $f^+$ or $f^-$ is absolutely continuous with respect to the Lebesgue measure $\mathcal{L}^d$. The $L^p$ summability of the transport density has been studied in several papers \cite{DePas1,DePas2,DePas3,San} and the results can be summarized as follows: for every $p \in [1,\infty]$, the transport density is in $L^p(\Omega)$ provided that both $f^+$ and $f^-$ are in $L^p(\Omega)$. However, none of these assumptions is satisfied in our setting: the source and target measures $f^+$ and $f^-$ in Problem \eqref{KantoIntro} are concentrated on the boundary, so it is not clear whether the transport density $|v|$ is unique or not. Nonetheless, the authors of \cite{DS} prove uniqueness of the optimal transport plan $\Lambda$ (and so, uniqueness of the optimal flow $v$) under the assumption that $f^+$ or $f^-$ is atomless and $\Omega$ is strictly convex. Moreover, they proved that if the domain is uniformly convex and both $f^+$ and $f^-$ are in $L^p(\partial\Omega)$ with $p \leq 2$, then the transport density between them is in $L^p(\Omega)$. Recalling again the equivalence between Problems \eqref{BV least gradient problem classical} \& \eqref{BeckmannIntro}, this implies that on uniformly convex domains the solution $u$ of the problem \eqref{BV least gradient problem classical} is unique as soon as $g \in C(\partial\Omega)$ and (for $p \leq 2$) it belongs to $W^{1,p}(\Omega)$ as soon as $g \in W^{1,p}(\partial\Omega)$.\\

In this paper, we are interested in studying the planar weighted least gradient problem (see \cite{JMN,MNT}): 
\begin{equation}\label{eq:weightedleastgradientproblem}\tag{wLGP}
\inf \bigg\{ \int_\Omega k(x)|Du|: \,\, u \in BV(\Omega), \,\, u|_{\partial\Omega} = g \bigg\}
\end{equation}
where $k:\Omega \mapsto \mathbb{R}^+$ is a given smooth function. It is closely connected to the conductivity imaging problem (see for instance \cite{JMN}), which is an inverse problem appearing with relation to medical imaging, and its purpose is to recover the conductivity $\sigma$ of a given body from a measurement of the magnitude of the current density $|J|$ inside the body and voltage $g$ on its boundary. Denote by $u$ the electrical potential corresponding to the voltage $g$, then $u$ formally satisfies the equation
\begin{equation}\label{eq:conductivityimaging}
\twopartdef{\mbox{div}(\sigma \nabla u) = 0}{\mbox{in }\, \Omega,}{u = g}{\mbox{on }\, \partial\Omega.}    
\end{equation} 
By Ohm's law, the current density equals $J = - \sigma \nabla u$, so we may formally rewrite the above problem as the weighted $1$-Laplace equation
\begin{equation}\label{eq:weighted1laplace}
\twopartdef{-\mbox{div}(|J| \frac{\nabla u}{|\nabla u|}) = 0}{\mbox{in }\, \Omega,}{u = g}{\mbox{on }\, \partial\Omega,}    
\end{equation}
which is the Euler-Lagrange equation of the the weighted least gradient problem with weight $k = |J|$ (for a precise justification of this passage, see \cite{MNT}). The conductivity imaging problem is in turn related to a variant of the Calder\'on problem, in which one wants to recover the conductivity of the body $\Omega$ solely from the measurement at the boundary and the knowledge of a Dirichlet-to-Neumann map associated to the weighted $1$-Laplace equation \eqref{eq:weighted1laplace} (or equivalently, to \eqref{eq:conductivityimaging}). Let us note that the Dirichlet-to-Neumann map for the isotropic $1$-Laplace operator was studied in \cite{HM}.

In \cite{JMN}, the authors showed existence of a solution $u$ for Problem \eqref{eq:weightedleastgradientproblem} as soon as $\partial\Omega$ satisfies a positivity condition on a sort of generalized mean curvature
related to $k$, and $g \in C(\partial\Omega)$. Moreover, they
proved that if $k \in C^{1,1}(\overline{\Omega})$ is positive and bounded away from zero and if $g$ is continuous
on $\partial\Omega$, then the weighted least gradient problem \eqref{eq:weightedleastgradientproblem} has at most one minimizer. They also showed that the condition $k \in C^{1,1}(\overline{\Omega})$ is sharp in the sense that uniqueness may
fail if $k \in  C^{1,\alpha}(\overline{\Omega})$ with $\alpha < 1$. In the present paper, we show existence and uniqueness of a solution $u$ for Problem \eqref{eq:weightedleastgradientproblem} via an optimal transport approach. More precisely, we will refine and generalize first the result of \cite{GRS2017NA} by proving that if $\Omega \subset \mathbb{R}^2$ is contractible and if $g \in BV(\partial\Omega)$, then Problem \eqref{eq:weightedleastgradientproblem} is completely equivalent to the weighted Beckmann problem: 
\begin{equation}\label{eq:weightedbeckmannproblem}\tag{wBP}
\inf \bigg\{ \int_{\overline{\Omega}} k(x)|v|: v \in \mathcal{M}(\overline{\Omega}, \mathbb{R}^2), \,\, \text{ div } v= 0, \,\, v \cdot n = f \text{ on } \partial\Omega \bigg\},
\end{equation}
where $f:= \partial_\tau g$. In particular, we have that $u$ is a solution for Problem \eqref{eq:weightedleastgradientproblem} if and only if the corresponding $v:=R_{\frac{\pi}{2}}Du$ is an optimal flow for Problem \eqref{eq:weightedbeckmannproblem} with $|v|(\partial\Omega)=0$. Yet, one can also show that Problem \eqref{eq:weightedbeckmannproblem} is equivalent to the following optimal transport problem with Riemannian cost (see also \cite{Prat,DweikWeighted}): 
\begin{equation}\label{Weighted KantoIntro}\min\bigg\{\int_{\overline{\Omega} \times \overline{\Omega}} d_k(x,y)\,\mathrm{d}\Lambda\,:\,\Lambda \in  \mathcal{M}^+(\overline{\Omega} \times \overline{\Omega}),\,(\Pi_x)_{\#}\Lambda=f^{+} \,\,\mbox{and}\,\,(\Pi_y)_{\#} \Lambda =f^{-}\bigg\},
\end{equation}
where $d_k$ denotes the distance
generated by the conformally  Riemannian metric $k$. 
In fact, we will show that if $\Omega$ is geodesically convex, then every optimal flow $v$ of Problem \eqref{eq:weightedbeckmannproblem} is of the form $v=k^{-1}v_\Lambda$, for some optimal transport plan $\Lambda$ of Problem \eqref{Weighted KantoIntro}, where
\begin{equation*} \label{Weighted transport density definition Intro}
<v_\Lambda,\xi>=\int_{\overline{\Omega} \times \overline{\Omega}} \int_0^1 k(\gamma_{x,y}(t))\,\xi(\gamma_{x,y}(t))\cdot \gamma_{x,y}^\prime(t)\,\mathrm{d}t\,\mathrm{d}\Lambda (x,y),\,\,\,\mbox{for all}\,\,\xi \in C(\overline{\Omega},\mathbb{R}^2).
\end{equation*}
It is not difficult to see that $|v_\Lambda|(\partial\Omega)=0$ as soon as $\Omega$ is geodesically strictly convex. Thanks to this fact, we will show that Problem \eqref{eq:weightedleastgradientproblem} reaches a minimum provided that $g \in BV(\partial\Omega)$. On the other hand, we will also show that Problem \eqref{Weighted KantoIntro} has a unique optimal transport plan $\Lambda$ as soon as $f^+$ or $f^-$ is atomless and $\Omega$ is geodesically strictly convex. This yields to uniqueness of the solution $u$ of Problem \eqref{eq:weightedbeckmannproblem} provided that $g \in C(\partial\Omega)$.

Moreover, we will study the $L^p$ summability of the transport density $|v_\Lambda|$ (where $\Lambda$ is the unique optimal transport plan in Problem \eqref{Weighted KantoIntro}) between two measures $f^+$ and $f^-$ which are located on the boundary $\partial\Omega$. More precisely, we will extend the $L^p$ estimates of \cite{DS} on the transport density to the case where the transport cost is given by a Riemannian distance $d_k(x,y)$, provided that $\Omega$ is  geodesically  uniformly convex. In fact, proving $L^p$ estimates on this transport density is significantly more difficult than in \cite{DS}, since the transport rays are now Riemannian geodesics; they are not necessarily straight lines like in the Euclidean case, and the key point in the proof would be to show a geometric result concerning the directions of the transport rays. Together with the equivalence between problems \eqref{eq:weightedleastgradientproblem} and \eqref{eq:weightedbeckmannproblem}, this yields $W^{1,p}$ regularity on the solution $u$ of Problem \eqref{eq:weightedleastgradientproblem}. In short, the main result in the present paper can be stated as follows: if $p \leq 2$, the domain $\Omega$ is geodesically uniformly convex and $g \in W^{1,p}(\partial\Omega)$, the weighted least gradient problem \eqref{eq:weightedleastgradientproblem} has a unique solution $u \in W^{1,p}(\Omega)$.\\

 The paper is organised as follows. In Section \ref{sec:preliminaries}, we recall the basic definitions concerning anisotropic BV functions and Anzellotti pairings, as well as the state of the literature concerning the weighted least gradient problem and the equivalence between the least gradient problem and the boundary-to-boundary optimal transport problem in the Euclidean case. In Section \ref{sec:lgpbeckequivalence}, we prove the equivalence between the anisotropic least gradient problem and anisotropic Beckmann problem under minimal assumptions. Then, in Section \ref{sec:beckmongeequivalence}, we consider Riemannian metrics and show equivalence between the weighted Beckmann problem and the optimal transport problem with Riemannian cost. Then, in Section \ref{sec:structure}, we study some consequences of the results in the previous two Sections for the structure of solutions in the least gradient problem. The main results in this paper, concerning $L^p$ estimates for the transport density, are in Section \ref{sec:lpestimates}. Finally, in Section \ref{sec:applications}, we apply the $L^p$ estimates for the transport density to obtain regularity estimates for solutions of the weighted least gradient problem.

{

\section{Preliminaries}\label{sec:preliminaries}

In this paper, our main goal is to study regularity of solutions to the weighted least gradient problem. To this end, we first recall in Section \ref{sec:anisotropicbv} the notion of anisotropic $BV$ spaces and some of their properties. We follow the presentation in \cite{AB}. We also recall the notion of Anzellotti pairings, weak normal traces and a corresponding version of the Gauss-Green formula (see \cite{Anz,CF}). Then, in Section \ref{sec:weightedlgp}, we recall the precise statement and known results on the weighted least gradient problem. Finally, in Section \ref{sec:equivalenceeuclidean}, we present known results on the equivalence between the two-dimensional least gradient problem and a version of the optimal transport problem in the Euclidean case. Throughout the whole paper, $\Omega \subset \mathbb{R}^2$ denotes an open bounded set with Lipschitz boundary.

\subsection{Anisotropic BV spaces}\label{sec:anisotropicbv}

\begin{definition}
A continuous function $\phi: \overline{\Omega} \times \mathbb{R}^2 \rightarrow [0, \infty)$ is called a metric integrand if it satisfies the following conditions:

{\flushleft $(1)$} $v \mapsto \phi(x,v)$ is a norm for every $x$, \\
$(2)$ $\phi$ is bounded and {uniformly} elliptic in $\overline{\Omega}$, i.e.
\begin{equation*}
\exists \,  \lambda,\, \Lambda  > 0 \quad \mbox{s.t.}\quad \forall \, x \in \overline{\Omega}, \quad \forall \, \xi \in \mathbb{R}^2, \quad \lambda |\xi| \leq \phi(x, \xi) \leq \Lambda |\xi|.
\end{equation*}
\end{definition}

Standard examples of metric integrands which appear with relation to the least gradient problem are: the Euclidean case (i.e. $\phi(x,p) = |p|$), the $l_1$ norm $\phi(x,p) = |p_1| + |p_2|$, and the weighted least gradient problem $\phi(x,p) = k(x) |p|$ with $k$ bounded and bounded away from zero (see \cite{Gor2018CVPDE,JMN,Zun}). This last case will be the main focus of our attention in the second half of the paper.

The definition below is specific to the two-dimensional case.

\begin{definition}
Let $\phi: \overline{\Omega} \times \mathbb{R}^2 \rightarrow [0, \infty)$ be a metric integrand. We denote its rotation by $-\frac{\pi}{2}$ in the second variable by
\begin{equation*}
\phi^\perp(x,\xi) := \phi(x, R_{-\frac{\pi}{2}} \xi).
\end{equation*}
It is clear that $\phi^\perp: \overline{\Omega} \times \mathbb{R}^2 \rightarrow [0, \infty)$ is also a metric integrand.
 
\end{definition}

\begin{definition}
Let $\phi: \overline{\Omega} \times \mathbb{R}^2 \rightarrow [0, \infty)$ be a metric integrand. Its polar function is $\phi^0: \overline{\Omega} \times \mathbb{R}^2 \rightarrow [0, \infty)$ defined as
\begin{equation*}
\phi^0 (x, \xi^*) = \sup \, \{ \langle \xi^*, \xi \rangle : \, \xi \in \mathbb{R}^2, \, \phi(x, \xi) \leq 1 \}.
\end{equation*} 
\end{definition}

\begin{definition}
Let $\phi: \overline{\Omega} \times \mathbb{R}^2 \rightarrow [0, \infty)$ be a metric integrand. For a given function $u \in L^1(\Omega)$, we define its $\phi-$total variation in $\Omega$ by the formula:

\begin{equation*}
\int_\Omega |Du|_\phi = \sup \, \bigg\{ \int_\Omega u \, \mathrm{div} (\mathbf{z}) \, dx : \, \phi^0(x,\mathbf{z}(x)) \leq 1 \, \, \, \text{a.e.},\,\, \mathbf{z} \in C_c^1(\Omega)  \bigg\}.
\end{equation*}
Another popular notation for the $\phi-$total variation is $\int_\Omega \phi(x, Du)$. In fact, one can show that $\int_\Omega \phi(x, Du)=\int_\Omega \phi(x,\frac{d Du}{d |Du|})\mathrm{d}|Du|$, where $\frac{d Du}{d |Du|}$ denotes the Radon-Nikodym derivative. We will say that $u \in BV_\phi(\Omega)$ if its $\phi-$total variation is finite {in $\overline{\Omega}$}; furthermore, let us define the $\phi-$perimeter of a set $E$ as
$$P_\phi(E, \Omega) = \int_{\Omega} |D\chi_E|_\phi.$$
If $P_\phi(E, \Omega) < \infty$, we say that $E$ is a set of bounded $\phi-$perimeter in $\Omega$.
\end{definition}

\begin{remark}
Let $\phi: \overline{\Omega} \times \mathbb{R}^2 \rightarrow [0, \infty)$ be a metric integrand. Since $\phi$ is bounded and uniformly elliptic, we have
$$\lambda \int_\Omega |Du| \leq \int_\Omega |Du|_\phi \leq \Lambda \int_\Omega |Du|.$$
In particular, $BV_\phi(\Omega) = BV(\Omega)$ as sets. They are equipped with different (but equivalent) norms.
\end{remark}

Now, we recall the definition and basic properties of Anzellotti pairings introduced in \cite{Anz}. The construction below is usually performed under more general assumptions, but for simplicity we restrict ourselves to the setting we will use in the paper (in Section \ref{sec:structure}). For $p \geq 1$, denote
\begin{equation*}
X_p(\Omega) = \bigg\{ \mathbf{z} \in L^\infty(\Omega, \mathbb{R}^2): \, \mbox{ div}(\mathbf{z}) \in L^p(\Omega) \bigg\}.
\end{equation*}
Given $\mathbf{z} \in X_p(\Omega)$ and $w \in BV(\Omega) \cap L^{p'}(\Omega)$, we define the functional $(\mathbf{z}, Dw): C_c^\infty(\Omega) \rightarrow \mathbb{R}$ by the formula
\begin{equation*}
\langle (\mathbf{z}, Dw), \varphi \rangle = - \int_\Omega w \, \varphi \, \mathrm{div}(\mathbf{z}) \, dx - \int_\Omega w \, \mathbf{z} \cdot \nabla \varphi \, dx.
\end{equation*}
The distribution $(\mathbf{z}, Dw)$ turns out to be a Radon measure on $\Omega$. It is a generalization of the pointwise product $\mathbf{z} \cdot \nabla w$; namely, for $w \in W^{1,1}(\Omega) \cap L^{p^\prime}(\Omega)$, we have
\begin{equation*}
\int_\Omega (\mathbf{z}, Dw)\,\varphi = \int_{\Omega} \mathbf{z} \cdot \nabla w \, \varphi \, dx, \quad \forall \, \varphi \in C^{\infty}_c(\Omega).
\end{equation*}
The following Proposition summarizes the most important properties of the pairing $(\mathbf{z}, Du)$.

\begin{proposition}\label{prop:boundonAnzellottipairing}
Suppose that $\phi$ is a metric integrand. Let $\mathbf{z} \in X_p(\Omega)$ and \,$u \in BV(\Omega) \cap L^{p'}(\Omega)$. Then, for any Borel set $B \subset \Omega$, we have
\begin{equation*}
\bigg| \int_{B} (\mathbf{z}, Du) \bigg| \leq \| \phi^0(x,\mathbf{z}(x)) \|_{L^\infty(\Omega)} \int_B |Du|_\phi.
\end{equation*} 
In particular, $(\mathbf{z}, Du) \ll |Du|$ as measures in $\Omega$. Moreover, there exists a function $[\mathbf{z},n] \in L^\infty(\partial\Omega)$ such that $\| [\mathbf{z},n] \|_{L^\infty(\partial\Omega)} \leq  \| \mathbf{z} \|_{L^\infty(\Omega,\mathbb{R}^2)}$ and the following {\it Gauss-Green formula} holds:
\begin{equation*}
\int_\Omega u \, \mathrm{div}(\mathbf{z}) \, dx + \int_\Omega (\mathbf{z}, Du) = \int_{\partial\Omega} [\mathbf{z},n]  u \, d\mathcal{H}^{1}.
\end{equation*}
\end{proposition}

The function $[\mathbf{z},n]$ has the interpretation of the normal trace of the vector field $\mathbf{z}$ at the boundary and it coincides with the classical normal trace if $\mathbf{z}$ is sufficiently smooth.

\subsection{Weighted least gradient problem}\label{sec:weightedlgp}

In this paper, we are interested in the following problem
\begin{equation}\label{wLGP-preliminaries}
\inf\bigg\{\int_{\Omega} k(x)|Du|\,:\,u \in BV(\Omega), \, u|_{\partial\Omega} = g \bigg\},
\end{equation}
where $k$ is sufficiently smooth. In the isotropic case, a standard assumption for existence of solutions is strict convexity of $\Omega$. In the anisotropic case, the standard assumption is called the barrier condition and is a local property at every point $x_0 \in \partial\Omega$. In the definition below, denote $\phi(x,p) = k(x) |p|$. Some of the results below are valid for more general $\phi$, but for the sake of presentation, we restrict ourselves to the two-dimensional weighted case.

\begin{definition}[Barrier condition]\label{Barrier Condition definition}
We say that $\Omega$ satisfies the barrier condition if the following condition holds: for every $x_0 \in \partial\Omega$, there exists $r_0 > 0$ such that for all $r < r_0$, if \,$V$ is a minimizer of
\begin{equation}\label{eq:barriercondition}
\inf \bigg\{ P_\phi(W, \mathbb{R}^2): \quad W \subset \Omega, \quad (\Omega \backslash W) \subset B(x_0,r) \bigg\},
\end{equation}
then
\begin{equation*}
\partial V \cap \partial \Omega \cap B(x_0,r) = \emptyset.
\end{equation*}
\end{definition}

Under this assumption, existence of solutions was proved for continuous boundary data by Jerrard, Moradifam and Nachman in \cite{JMN} (actually, the result holds also when the boundary datum is continuous a.e. with respect to the codimension one Hausdorff measure on $\partial\Omega$, see \cite{Mor} or \cite{Gor2021IUMJ}). However, in order to use optimal transport techniques, we will later need a slighly stronger assumption of strict geodesic convexity of $\Omega$. In the Euclidean case, these two concepts are closely related: in two dimensions, the barrier condition is equivalent to strict convexity of $\Omega$; in higher dimensions, the barrier condition is something between convexity and strict convexity of $\Omega$, see \cite{Gor2021IUMJ}. We will comment on the relationship between these two concepts at the end of Section \ref{sec:beckmongeequivalence}.

The literature regarding regularity of solutions to the weighted least gradient problem is very thin. In the paper \cite{JMN}, the authors have shown that for sufficiently smooth $k$, in low dimensions continuity of boundary data implies continuity of solutions. This result was extended to any dimension by Zuniga in \cite{Zun} under the assumption $k \in C^2(\overline{\Omega})$. However, if $k$ is not regular enough, then continuity of solutions may break down: several examples for Lipschitz weights appear in \cite{LMSS} and an example for $C^{1,\alpha}$ weights with $\alpha < 1$ appears in \cite{JMN}.

Finally, let us recall another formulation of the least gradient problem, first introduced by Maz\'on, Rossi and Segura de Le\'on in the isotropic case (see \cite{MRL}). 

\begin{definition}
We say that $u \in BV(\Omega)$ is a function of $\phi$-least gradient if for all $v \in BV(\Omega)$ such that $v|_{\partial\Omega}=0$, we have
$$ \int_\Omega |Du|_\phi \leq \int_\Omega |D(u+v)|_\phi.$$
\end{definition}

Using this definition, we may reformulate the weighted least gradient problem as follows: the goal is to find a $\phi$-least gradient function with prescribed trace. The following characterisation in terms of Anzellotti pairings was proved in \cite{MRL} in the isotropic case; in the form below it was proved by Maz\'on in \cite{Maz}.

\begin{theorem}\label{thm:mrlcharacterisation}
Suppose that $u \in BV(\Omega)$ is a function of $\phi$-least gradient. Then, there exists a vector field \,$\mathbf{z} \in L^\infty(\Omega, \mathbb{R}^2)$ such that $\| \phi^0(x,\mathbf{z}(x)) \|_\infty \leq 1$, $\mbox{div}(\mathbf{z}) = 0$ in the distributional sense and 
$$ (\mathbf{z},Du) = |Du|_\phi  \qquad \mbox{ as measures in } \Omega. $$
\end{theorem}

We also recall the celebrated Bombieri-de Giorgi-Giusti theorem. It was first proved in the Euclidean case in \cite{BGG} and in the form below it was proved by Maz\'on in \cite{Maz}.

\begin{theorem}\label{thm:anisotropicbgg}
Suppose that $u \in BV(\Omega)$ is a $\phi$-least gradient function. Then, for all $t \in \mathbb{R}$, the function $\chi_{\{ u \geq t \}}$ is also a function of $\phi-$least gradient.
\end{theorem}

\subsection{Equivalence in the Euclidean case}\label{sec:equivalenceeuclidean}

In this Section, we present the relationship between the least gradient problem, the Beckmann problem and the classical Monge-Kantorovich problem in the Euclidean case, together with some properties of solutions to these problems. Recall that throughout the paper $\Omega \subset \mathbb{R}^2$ denotes an open bounded set with Lipschitz boundary. In this subsection, we assume additionally that $\Omega$ is convex.
We start with the relationship between the least gradient problem
\begin{equation}\label{eq:leastgradientproblem}
\inf\bigg\{ \int_\Omega |Du| \,:\,u \in BV(\Omega),\, u|_{\partial\Omega} = g \bigg\},
\end{equation}
and the Beckmann problem
\begin{equation}\label{eq:beckmannproblem}
\min \bigg\{ \int_{\overline{\Omega}} |v|\,:\, v \in \mathcal{M}(\overline{\Omega}, \mathbb{R}^2),\,\, \mathrm{div} \, v = f \bigg\}.
\end{equation}
Here, $g \in BV(\partial\Omega)$, and the boundary condition in \eqref{eq:leastgradientproblem} is understood in the sense of traces. The divergence condition in \eqref{eq:beckmannproblem} is understood in the distributional sense; we have $\mathrm{div} \, (v) = 0$ in $\Omega$ and $[v, n] = f$ on $\partial\Omega$, where $[v,n]$ denotes the weak normal trace of a vector field whose divergence is integrable (see Section \ref{sec:anisotropicbv}). Suppose that $f \in \mathcal{M}(\partial\Omega)$ satisfies a mass balance condition  $\int_{\partial\Omega} \mathrm{d}f = 0$. Then, Problem \eqref{eq:beckmannproblem} admits a solution (see \cite{San2015}). Moreover, Problem \eqref{eq:leastgradientproblem} reaches a minimum as soon as $\Omega$ is strictly convex (see \cite{Gor2018CVPDE}). Notice that such $f$ and $g$ are in a one-to-one correspondence (up to an additive constant in $g$) via the relation $f = \partial_\tau g$. It was observed in \cite{GRS2017NA} that Problems \eqref{eq:leastgradientproblem} and \eqref{eq:beckmannproblem} are closely related. Namely, if we take an admissible function $u \in BV(\Omega)$ in \eqref{eq:leastgradientproblem}, then $v = R_{\frac{\pi}{2}} Du$ is admissible in \eqref{eq:beckmannproblem}. Indeed, in dimension two, a rotation of a gradient by $\frac{\pi}{2}$ is a divergence-free field in $\Omega$ and it interchanges the normal and tangent components at the boundary. In the other direction, given a vector field $v \in L^1(\Omega,\mathbb{R}^2)$ admissible in \eqref{eq:beckmannproblem}, we can recover $u \in W^{1,1}(\Omega)$ admissible in \eqref{eq:leastgradientproblem}. In \cite{DS}, the authors noticed that this result may be improved: if $v \in \mathcal{M}(\overline{\Omega}, \mathbb{R}^2)$ is such that $|v|(\partial\Omega) = 0$, then there exists $u \in BV(\Omega)$ such that $v = R_{\frac{\pi}{2}} Du$ and $[v,n] = \partial_\tau(Tu)$. In particular, notice that $|v| = |Du|$ as measures on $\overline{\Omega}$; hence, the infimal values in \eqref{eq:leastgradientproblem} and \eqref{eq:beckmannproblem} coincide. We will see later that a solution $v$ for Problem \eqref{eq:beckmannproblem} gives zero mass to the boundary as soon as $\Omega$ is strictly convex (this implies existence of a solution for Problem \eqref{eq:leastgradientproblem}; see \cite{DS}). Let us sum up these considerations as follows.

\begin{theorem}\label{thm:lgpbeckmannequivalence} Suppose that $\Omega$ is convex. Then, the problems \eqref{eq:leastgradientproblem} and \eqref{eq:beckmannproblem} are equivalent in the following sense: \\
(1) Their infimal values coincide, i.e. $\inf \eqref{eq:leastgradientproblem} = \min \eqref{eq:beckmannproblem}$. \\
(2) Given a solution $u \in BV(\Omega)$ of \eqref{eq:leastgradientproblem}, we can construct a solution $v \in \mathcal{M}(\overline{\Omega}, \mathbb{R}^2)$ of \eqref{eq:beckmannproblem}; moreover, $v = R_{\frac{\pi}{2}} Du$. \\
(3) Given a solution $v \in \mathcal{M}(\overline{\Omega}, \mathbb{R}^2)$ of \eqref{eq:beckmannproblem} with $|v|(\partial\Omega) = 0$, we can construct a solution $u \in BV(\Omega)$ of \eqref{eq:leastgradientproblem}; moreover, $v = R_{\frac{\pi}{2}} Du$.
\end{theorem}

Now, we turn to the equivalence between the Beckmann problem \eqref{eq:beckmannproblem}
and the following Monge-Kantorovich problem:
\begin{equation}\label{eq:kantorovich}
\min \bigg\{ \int_{\overline{\Omega} \times \overline{\Omega}} |x-y| \, \mathrm{d}\Lambda\, : \, \Lambda \in \mathcal{M}^+(\overline{\Omega} \times \overline{\Omega}), \, (\Pi_x)_{\#}\Lambda = f^+ \,\, \mathrm{and} \,\, (\Pi_y)_{\#} \Lambda = f^- \bigg\}.
\end{equation}
This equivalence is standard in the optimal transport theory (see for instance \cite[Chapter 4]{San2015}), but we will describe it shortly for completeness. Here, $\Omega \subset \mathbb{R}^2$ is a bounded open convex set, $f \in \mathcal{M}(\overline{\Omega})$ and $f = f^+ - f^-$ is its decomposition into a positive and negative part. In order for the problem to be well defined, we assume that the mass balance condition $f^+(\overline{\Omega}) = f^-(\overline{\Omega})$ holds. In particular, this setting covers our case, when the measures are concentrated on $\partial\Omega$.

First, let us recall a few standard results in the optimal transport theory. The fact that 
\begin{equation}\label{eq:dualisotropic}
\sup\bigg\{\int_{\overline{\Omega}} \psi\,\mathrm{d}(f^+ - f^-)\,:\,\psi \in  \mathrm{Lip}_1(\overline{\Omega})\bigg\} 
\end{equation}
is the dual problem to \eqref{eq:kantorovich} is well-known, see for instance \cite{San2015,Vil}. As a corollary of its proof, we get that there exist solutions to both problems, and that any optimal transport plan $\Lambda$ and any Kantorovich potential $\psi$ satisfy the following equality:
$$\int_{\overline{\Omega} \times \overline{\Omega}} (|x-y| - (\psi(x) - \psi(y))) \, \mathrm{d}\Lambda(x,y) =0,$$
which implies that 
$$\psi(x) - \psi(y) = |x-y| \quad \mbox{on\, spt}(\Lambda).$$ 
If $\psi$ is a Kantorovich potential, we call any maximal segment \,$\left[x,y\right]$ satisfying \,$\psi(x)-\psi(y)=|x-y|$ a 
{\it transport ray}. The Kantorovich potential is in general not unique, but frame of transport rays does not depend on the choice of $\psi$ (at least on the support of $\Lambda$). Moreover, an optimal transport plan $\Lambda$ has to move the mass along the transport rays.

Now, we turn our attention to the equivalence between the Kantorovich problem \eqref{eq:kantorovich} and the Beckmann problem \eqref{eq:beckmannproblem}. First, let us see that for any $v \in \mathcal{M}(\overline{\Omega},\mathbb{R}^d)$ admissible in the Beckmann problem \eqref{eq:beckmannproblem}, we have that for any $C^1$ function $\phi$ with $|\nabla \phi| \leq 1$,
\begin{equation*}
|v|(\overline{\Omega}) = \int_{\overline{\Omega}} 1 \, \mathrm{d}|v| \geq \int_{\overline{\Omega}} (-\nabla \phi) \cdot \mathrm{d}v = \int_{\overline{\Omega}} \phi \, \mathrm{d}f.
\end{equation*}
Hence, $\min \eqref{eq:beckmannproblem} \geq \max \eqref{eq:dualisotropic} = \min \eqref{eq:kantorovich}$ (the supremum in \eqref{eq:dualisotropic} is taken for Lipschitz functions, but we may approximate them uniformly by $C^1$ functions). On the other hand, given an optimal transport plan $\Lambda$, we may construct a vector measure $v_\Lambda \in \mathcal{M}(\overline{\Omega},\mathbb{R}^d)$ defined by the formula
\begin{equation*}\label{eq:definitionofpgamma}
\langle v_\Lambda,\xi \rangle :=\int_{\overline{\Omega} \times \overline{\Omega}} \int_0^1  \xi(\omega_{x,y}(t)) \cdot \omega_{x,y}'(t)  \, \mathrm{d}t \, \mathrm{d}\Lambda (x,y),
\end{equation*}
for all $\xi \in C(\overline{\Omega}, \mathbb{R}^d)$. Here, $\omega_{x,y}(t) = (1-t)x + ty$ is the constant-speed parametrisation of $[x,y]$. By taking $\xi = \nabla \phi$, it is immediate that $v_\Lambda$ satisfies the divergence constraint. The total mass of $v_\Lambda$ will be estimated using the {\it transport density} $\sigma_\Lambda \in \mathcal{M}^+(\overline{\Omega})$, which is defined by the formula
\begin{equation*}\label{eq:definitionofsigma}
\langle \sigma_\Lambda,\phi \rangle := \int_{\overline{\Omega} \times \overline{\Omega}} \int_0^1   \phi (\omega_{x,y}(t)) |\omega_{x,y}'(t)|\, \mathrm{d}t \, \mathrm{d}\Lambda (x,y),
\end{equation*}
for all $\phi \in C(\overline{\Omega})$. The vector measure $v_\Lambda$ and the scalar measure $\sigma_\Lambda$ are related in the following way: if $\psi$ is a Kantorovich potential, we have
\begin{equation*}
\omega_{x,y}'(t) = y - x = - |x-y| \frac{x-y}{|x-y|} = - |x-y| \nabla \psi(\omega_{x,y}(t)),
\end{equation*}
for all $t \in (0,1)$ and $(x,y) \in \mathrm{spt}(\Lambda)$. Thus, $\langle v_\Lambda, \xi \rangle = \langle \sigma_\Lambda, -\xi \cdot \nabla \psi \rangle$, for all $\xi \in C(\overline{\Omega},\mathbb{R}^d)$, so we have 
\begin{equation*}
v_\Lambda = -\nabla \psi \cdot \sigma_\Lambda.
\end{equation*}
Hence, $v_\Lambda$ is absolutely continuous with respect to $\sigma_\Lambda$ and $|v_\Lambda| = \sigma_\Lambda$. Then, one has
\begin{equation*}
\min \eqref{eq:kantorovich} = \int_{\overline{\Omega} \times \overline{\Omega}} |x-y| \, \mathrm{d}\Lambda = \int_{\overline{\Omega} \times \overline{\Omega}} \int_0^1 |\omega_{x,y}'(t)| \, \mathrm{d}t \, \mathrm{d}\Lambda(x,y) = \sigma_\Lambda(\overline{\Omega}) = |w_\Lambda|(\overline{\Omega}) \geq \min \eqref{eq:beckmannproblem}.
\end{equation*}
This implies that $\min \eqref{eq:kantorovich} = \min \eqref{eq:beckmannproblem}$ and that from an optimal transport plan $\Lambda$ we can construct a solution to the Beckmann problem \eqref{eq:beckmannproblem}. 
On the other hand, it can be shown that every solution to \eqref{eq:beckmannproblem} is of the form $v = v_\Lambda$ for some optimal transport plan $\Lambda$, see \cite[Theorem 4.13]{San2015}. We summarize the above discussion in the following Theorem.

\begin{theorem}\label{thm:beckmannkantorovichequivalence} Suppose that $\Omega$ is convex. Then, the problems \eqref{eq:kantorovich} and \eqref{eq:beckmannproblem} admit solutions, and are equivalent in the following sense: \\
(1) Their minimal values coincide, i.e. $\min \eqref{eq:kantorovich} = \min \eqref{eq:beckmannproblem}$. \\
(2) Given an optimal transport plan $\Lambda \in \mathcal{M}^+(\overline{\Omega} \times \overline{\Omega})$ in \eqref{eq:kantorovich}, we can construct a solution $v_\Lambda \in \mathcal{M}(\overline{\Omega}, \mathbb{R}^d)$ to \eqref{eq:beckmannproblem}.
\\
(3) Given a solution $v \in \mathcal{M}(\overline{\Omega}, \mathbb{R}^2)$ to \eqref{eq:beckmannproblem}, we can construct an optimal transport plan $\Lambda \in \mathcal{M}^+(\overline{\Omega} \times \overline{\Omega})$ in \eqref{eq:kantorovich} such that $v = v_\Lambda$.
\end{theorem}

The equivalence between the two-dimensional least gradient problem \eqref{eq:leastgradientproblem} and the Monge-Kantorovich problem \eqref{eq:kantorovich} comes from combining Theorems \ref{thm:lgpbeckmannequivalence} and \ref{thm:beckmannkantorovichequivalence}. Given a solution to \eqref{eq:leastgradientproblem}, we may construct an optimal transport plan $\Lambda$ for \eqref{eq:kantorovich} with $f^\pm = (\partial_\tau g)^\pm$; in the other direction, since $|v_\Lambda| = \sigma_\Lambda$, we may recover a solution to \eqref{eq:leastgradientproblem} from an optimal transport plan $\Lambda$ as soon as the transport density $\sigma_\Lambda$ gives no mass to the boundary, i.e. $\sigma_\Lambda(\partial\Omega) = 0$. An important special case is when $\Omega$ is strictly convex. Using the equivalent formula for transport density, i.e. for every Borel set $A \subset \overline{\Omega}$,
\begin{equation*}\label{eq:definitionofsigmaversiontwo}
\sigma_\Lambda(A) = \int_{\overline{\Omega} \times \overline{\Omega}} \mathcal{H}^1([x,y] \cap A) \, \mathrm{d}\Lambda(x,y),
\end{equation*}
we have that if $\Omega$ is strictly convex, then for any optimal transport plan $\Lambda$ we have that $\sigma_\Lambda(\partial\Omega) = 0$, and the correspondence between problems is one-to-one. This link between the two problems was exploited for the first time in \cite{DS}. For a strictly convex domain $\Omega$, the authors studied the boundary-to-boundary optimal transport problem. If at least one of the measures $f^\pm$ is atomless, then the optimal transport plan is unique and induced by a map, and proved several variants of regularity estimates on the transport density. The main application of these results so far is the $W^{1,p}$ regularity of solutions to the least gradient problem: the estimates on the transport density imply that if $g \in W^{1,p}(\partial\Omega)$ for $p \in [1,2]$, then the unique solution $u$ to \eqref{eq:leastgradientproblem} lies in $W^{1,p}(\Omega)$, provided that $\Omega$ is uniformly convex.

}

\section{Equivalence between the anisotropic least gradient problem and the anisotropic Beckmann problem}\label{sec:lgpbeckequivalence}
  
{ In this Section, we will prove the equivalence between the anisotropic least gradient problem 
\begin{equation*}
\inf\bigg\{ \int_\Omega \phi(x,Du) \,:\,u \in BV(\Omega),\, u|_{\partial\Omega} = g \bigg\},
\end{equation*}
and the anisotropic Beckmann problem
\begin{equation*}
\min \bigg\{ \int_{\overline{\Omega}} \phi^\perp(x,v) \,:\, v \in \mathcal{M}(\overline{\Omega}, \mathbb{R}^2),\,\, \mathrm{div} \, v = f \bigg\}. 
\end{equation*}
Recall that throughout the paper $\Omega \subset \mathbb{R}^2$ denotes an open bounded set with Lipschitz boundary. In this Section, we will additionally require that $\Omega$ is contractible (or equivalently, since we study the planar case, that $\partial\Omega$ is connected). This is unlike the results in \cite{GRS2017NA} for the Euclidean case, which require convexity of the domain. We will generalize the results from \cite{GRS2017NA} in two ways: by allowing anisotropic settings and a wider class of domains. Relaxing the convexity assumption on $\Omega$ is particularly relevant in the anisotropic case; throughout most of the paper, we will instead require that the domain $\Omega$ is geodesically convex}. Hence, we need to prove the equivalence for a more general class of domains.


\begin{proposition}\label{prop:L1recovery}
{ Suppose that $\Omega$ is contractible.} Let $v\in L^1(\Omega, \mathbb{R}^2)$ be such that $\mathrm{div} (v) =0$ in the sense of distributions. Then, there exists $u \in W^{1,1}(\Omega)$ such that $v=R_{\frac{\pi}{2}} \nabla u$. In particular, for any metric integrand $\phi$, we have
$$ \int_\Omega \phi^\perp(x,v) \,dx = \int_\Omega \phi(x,\nabla u) \,dx. $$
Moreover, if \,$v\cdot n|_{\partial\Omega} = f$ and \,$Tu = g$, then we have $f=\partial_\tau g$.
\end{proposition}

\begin{proof} 

{\bf Step 1. Definition of an approximating sequence $u^\varepsilon$.} Fix any $x_0 \in \Omega$. Denote $v=(v_1,v_2)$. Let $\varphi_\varepsilon$ (with $\varepsilon > 0$) be a sequence of mollifiers and set $v^\varepsilon$ to be the mollification of $v$, i.e. $v^\varepsilon = v\ast\varphi_\varepsilon:=(v_1^\varepsilon,v_2^\varepsilon)$. Since $\mathrm{div} (v)=0$, we also have $\mathrm{div} (v^\varepsilon) = 0$. We define a sequence of mollified differential forms $\omega^\varepsilon = v^\varepsilon_2 dx_1- v^\varepsilon_1 dx_2$; since $\mathrm{div}(v^\varepsilon) = 0$, we also have $d\omega^\varepsilon = 0$. The domain is contractible, so the differential form $\omega^\varepsilon$ uniquely (up to an additive constant depending on the choice of $x_0$) defines a function $u^\varepsilon \in C^\infty(\overline{\Omega})$ via
$$ u^\varepsilon(x) = \int_{\gamma_x} \omega^\varepsilon, $$
with $\gamma_x$ is any smooth oriented curve from $x_0$ to $x$. In particular, we have $\nabla u^\varepsilon=R_{-\frac\pi2}v^\varepsilon$. 
 
{\bf Step 2. Convergence in $L^1(\Omega)$.} We claim that a modified version of the sequence $u^\varepsilon$ converges in $L^1(\Omega)$. 
Denote by $(u^\varepsilon)_\Omega$ the mean value of $u^\varepsilon$ on $\Omega$. Set $\widetilde{u^\varepsilon} = u^\varepsilon - (u^\varepsilon)_\Omega$ and notice that $\nabla \widetilde{u^\varepsilon} = \nabla u^\varepsilon = R_{-\frac\pi2}v^\varepsilon$. Using the Poincar\'e inequality, we get the following estimate:
\begin{equation}\label{eq:L1approximationestimate}
\int_\Omega |\widetilde{u^\varepsilon}| \, dx = \int_\Omega |u^\varepsilon - (u^\varepsilon)_\Omega| \, dx \leq C(\Omega) \int_\Omega |\nabla u^\varepsilon| \, dx = C(\Omega) \int_\Omega |v^\varepsilon| \, dx.
\end{equation}
Yet, the sequence $v^\varepsilon$ is bounded in $L^1(\Omega,\mathbb{R}^2)$ with $||v^\varepsilon||_{L^1} \leq ||v||_{L^1}$. Hence, we have
\begin{equation*}
\| \widetilde{u^\varepsilon} \|_{W^{1,1}(\Omega)} \leq (C(\Omega) + 1) \| v^\varepsilon \|_{L^1(\Omega)} \leq M.
\end{equation*}
Then, there exists a subsequence $\widetilde{u^{\varepsilon_n}} \rightarrow u$ in $L^1(\Omega)$.

{\bf Step 3. Convergence in $W^{1,1}(\Omega)$.} We recall that $\nabla \widetilde{u^\varepsilon} = R_{-\frac\pi2}v^\varepsilon$. Since $v^\varepsilon$ is a Cauchy sequence in $L^1(\Omega, \mathbb{R}^2)$, $\nabla \widetilde{u^{\varepsilon_n}}$ is also a Cauchy sequence in $L^1(\Omega, \mathbb{R}^2)$, because
\begin{equation*}
\| \nabla \widetilde{u^{\varepsilon_n}} - \nabla \widetilde{u^{\varepsilon_m}} \|_{L^1(\Omega, \mathbb{R}^2)} = \| R_{-\frac\pi2}(v^{\varepsilon_n} -  v^{\varepsilon_m}) \|_{L^1(\Omega, \mathbb{R}^2)} = \| v^{\varepsilon_n} - v^{\varepsilon_m} \|_{L^1(\Omega,\mathbb{R}^2)}.     
\end{equation*}
Hence, $\widetilde{u^{\varepsilon_n}}$ is a Cauchy sequence in $W^{1,1}(\Omega)$, so it converges to $u$ in $W^{1,1}(\Omega)$.
In particular,
we have
$\nabla u=R_{-\frac\pi2}v$ and,
$$ \int_\Omega \phi(x,\nabla u) \, dx = 
\int_\Omega \phi^\perp(x,v)\, dx.$$

{\bf Step 4. The boundary condition.} Finally, we study the relationship between the trace of $f$ and the normal trace of $g$; this part is similar to part of the proof in \cite{GRS2017NA}, but we write it here for completeness. If $\phi \in \mbox{Lip}(\mathbb{R}^2)$, and $\varphi$ is its restriction on $\partial\Omega$, then
\begin{eqnarray*}
\langle f,\varphi\rangle &= & \langle v\cdot n|_{\partial\Omega}, \varphi \rangle = \int_\Omega v \cdot \nabla\phi \,dx = 
\lim_{\varepsilon_n \rightarrow 0} \int_\Omega v^\varepsilon \cdot \nabla\phi \,dx =
\lim_{\varepsilon_n \rightarrow 0} \int_\Omega R_{\frac{\pi}{2}} \nabla \widetilde{u^{\varepsilon_n}} \cdot \nabla\phi \,dx\\
&= & \lim_{\varepsilon_n \rightarrow 0} \int_{\partial\Omega} R_{\frac{\pi}{2}} \nabla \widetilde{u^{\varepsilon_n}} \cdot n\, \varphi\,d\mathcal{H}^1.
\end{eqnarray*}
Due to the smoothness of $\widetilde{u^{\varepsilon_n}}$, we have 
$R_{\frac{\pi}{2}} \nabla \widetilde{u^{\varepsilon_n}} \cdot n = {\partial_\tau \widetilde{u^{\varepsilon_n}}}$. Hence,
\begin{equation*}
 \langle f,\varphi\rangle = \lim_{\varepsilon_n \rightarrow 0} \int_{\partial\Omega} R_{\frac{\pi}{2}} \nabla \widetilde{u^{\varepsilon_n}} \cdot n \,\varphi\,d\mathcal{H}^1
 = \lim_{\varepsilon_n \rightarrow 0} \int_{\partial\Omega} {\partial_\tau \widetilde{u^{\varepsilon_n}}}\, \varphi\,d\mathcal{H}^1 =
 -\lim_{\varepsilon_n \rightarrow 0} \int_{\partial\Omega} T\widetilde{u^{\varepsilon_n}}\,{\partial_\tau \varphi}\,d\mathcal{H}^1. 
\end{equation*}
Since $\widetilde{u^{\varepsilon_n}}$ converges to $u$ in $W^{1,1}(\Omega)$, then the traces converge in $L^1(\partial\Omega)$, i.e., $T \widetilde{u^{\varepsilon_n}} \rightarrow Tu$.
Thus,
\begin{equation*}
 \langle f,\varphi\rangle =- \int_{\partial\Omega} Tu\, {\partial_\tau \varphi}\,\,d\mathcal{H}^1 =- \int_{\partial\Omega} g\, {\partial_\tau \varphi}\,\,d\mathcal{H}^1
 =\left\langle\partial_\tau g, \varphi\right\rangle. \qedhere
\end{equation*}
\end{proof}

In a similar way, we may enlarge the class of admissible vector fields $v$. We will require that $v \in \mathcal{M}(\overline{\Omega}, \mathbb{R}^2)$ is such that $|v|(\partial\Omega) = 0$. The outline of this proof will be very similar to the proof of Proposition \ref{prop:L1recovery}.

\begin{proposition}\label{prop:measurerecovery}
{ Suppose that $\Omega$ is contractible.} Let $v\in \mathcal{M}(\overline{\Omega}, \mathbb{R}^2)$ be such that $\mathrm{div} (v) =0$ in the sense of distributions and $|v|(\partial\Omega) = 0$. Then, there exists $u \in BV(\Omega)$ such that $v=R_{\frac{\pi}{2}} Du$. In particular, for any metric integrand $\phi$, we have
$$ \int_\Omega \phi^\perp(x,v) \,dx = \int_\Omega \phi(x,Du) \,dx.$$
In addition, if \,$v\cdot n|_{\partial\Omega} = f$ and \,$Tu = g$, then we have $f=\partial_\tau g$.
\end{proposition}

\begin{proof} 
{\bf Step 1. Definition of an approximating sequence $u^\varepsilon$.} We proceed exactly as in the proof of Proposition \ref{prop:L1recovery}; we also keep the same notation. 

{\bf Step 2. Convergence in $L^1(\Omega)$.} The estimate \eqref{eq:L1approximationestimate} is proved on the level of the approximation $\widetilde{u^\varepsilon}$ and remains the same when $v \in \mathcal{M}(\overline{\Omega}, \mathbb{R}^2)$. 
When $v$ is a measure, we also have 
$v^\varepsilon$ is bounded in $L^1(\Omega, \mathbb{R}^2)$ with $||v^\varepsilon||_{L^1(\Omega, \mathbb{R}^2)} \leq ||v||_{\mathcal{M}(\overline{\Omega},\mathbb{R}^2)}$. Hence
$$ \| \widetilde{u^\varepsilon} \|_{W^{1,1}(\Omega)} \leq M.$$
Then, there exists a subsequence $\widetilde{u^{\varepsilon_n}}$ such that $\widetilde{u^{\varepsilon_n}} \rightarrow u$ weakly* in $BV(\Omega)$, i.e. there is a function $u \in BV(\Omega)$ such that $\widetilde{u^{\varepsilon_n}} \rightarrow u$ in $L^1(\Omega)$ and $\nabla \widetilde{u^{\varepsilon_n}} \, d\mathcal{L}^2 \rightharpoonup Du$ weakly* as measures. In particular, we have $v=R_{\frac{\pi}{2}} Du$.
 
{\bf Step 3. Convergence of the anisotropic total variations.} Now, recall that $|v|(\partial\Omega) = 0$; then, also $|\phi^\perp(x,v)|(\partial\Omega) = 0$. 
Hence, one has
$$  \int_\Omega \phi(x,Du) = \int_\Omega \phi^\perp(x,v).$$
Moreover, we have
$$ \int_\Omega |v| = \lim_{\varepsilon_n \rightarrow 0} \int_\Omega |v^{\varepsilon_n}| \, dx = \lim_{\varepsilon_n \rightarrow 0} \int_\Omega |\nabla \widetilde{u^{\varepsilon_n}}| \, dx \geq \int_\Omega |Du| = \int_\Omega |v|,$$
and so, the inequality is in fact an equality and the subsequence $\widetilde{u^{\varepsilon_n}} \rightarrow u$ 
strictly in $BV(\Omega)$. 

{\bf Step 4. The boundary condition.} This step is handled analogously to the corresponding part in the proof of Proposition \ref{prop:L1recovery}. The only difference is that the convergence $T\widetilde{u^{\varepsilon_n}} \rightarrow Tu$ follows from strict convergence of the sequence $\widetilde{u^{\varepsilon_n}}$ to $u$ in place of convergence in norm.
\end{proof}

{ Finally, as a consequence of Step 4 in the proof of Proposition \ref{prop:L1recovery}, we get a converse result to Proposition \ref{prop:L1recovery}. Actually, the next result is valid without the contractibility assumption on $\Omega$, and a proof can be found in \cite[Proposition 3.1]{DG2019}.}

\begin{proposition}\label{prop:admissibility}
Let $u \in BV(\Omega)$ with trace $Tu = g$. Then, $v = R_{\frac{\pi}{2}}\nabla u$ is a vector-valued measure such that $\mathrm{div} (v) = f$, where $f = \partial_\tau g$. In particular, it is an admissible function in \eqref{eq:beckmannproblem}. 
\end{proposition} 

Now, we state the main Theorem.

\begin{theorem}\label{thm:anisotropicequivalence} { Suppose that $\Omega$ is contractible.} Then, the problems \eqref{eq:leastgradientproblem} and \eqref{eq:beckmannproblem} are equivalent in the following sense: \\
(1) Their infimal values coincide, i.e. $\inf \eqref{eq:leastgradientproblem} = \min \eqref{eq:beckmannproblem}$. \\
(2) Given a solution $u \in BV(\Omega)$ of \eqref{eq:leastgradientproblem}, we can construct a solution $v \in \mathcal{M}(\overline{\Omega}, \mathbb{R}^2)$ of \eqref{eq:beckmannproblem}; moreover, $v = R_{\frac{\pi}{2}} Du$. \\
(3) Given a solution $v \in \mathcal{M}(\overline{\Omega}, \mathbb{R}^2)$ of \eqref{eq:beckmannproblem} with $|v|(\partial\Omega) = 0$, we can construct a solution $u \in BV(\Omega)$ of \eqref{eq:leastgradientproblem}; moreover, $v = R_{\frac{\pi}{2}} Du$.
\end{theorem}


\begin{proof}
(1) Suppose that $(u_n)_n \subset BV(\Omega)$ is a minimizing sequence for the problem \eqref{eq:leastgradientproblem}. Then, taking $v_n = R_{\frac{\pi}{2}} Du_n$, we have by Proposition \ref{prop:admissibility} that the functions $v_n$ are admissible in Problem \eqref{eq:beckmannproblem} and we have
\begin{equation*}
\inf \eqref{eq:leastgradientproblem} \leftarrow \int_\Omega \phi(x,Du_n) = \int_\Omega \phi^\perp(x,v_n) \geq \min \eqref{eq:beckmannproblem}.
\end{equation*}
Similarly, suppose that $v_n \in \mathcal{M}(\overline{\Omega}, \mathbb{R}^2)$ is a minimizing sequence for the problem \eqref{eq:beckmannproblem}. Mollifying this sequence if necessary, we may take another minimizing sequence $\widetilde{v_n} \in L^1(\Omega, \mathbb{R}^2)$ for this problem. By Proposition \ref{prop:L1recovery}, there exist functions $u_n \in W^{1,1}(\Omega)$ admissible in Problem \eqref{eq:leastgradientproblem} such that $\widetilde{v_n} = R_{\frac{\pi}{2}} \nabla u_n$. Hence,
\begin{equation*}
\min \eqref{eq:beckmannproblem} \leftarrow \int_\Omega \phi^\perp(x,\widetilde{v_n}) = \int_\Omega \phi(x,\nabla u_n) \geq \inf \eqref{eq:leastgradientproblem}.
\end{equation*}
We obtained that $\inf \eqref{eq:leastgradientproblem} = \min \eqref{eq:beckmannproblem}$.

(2) Let $u \in BV(\Omega)$ be a minimizer of Problem \eqref{eq:leastgradientproblem}. Let $v = R_{\frac{\pi}{2}} \nabla u$; by Proposition \ref{prop:admissibility}, it is an admissible vector field in Problem \eqref{eq:beckmannproblem}. We have
$$ \inf \eqref{eq:leastgradientproblem}=
\int_\Omega \phi(x,Du) = \int_\Omega \phi^\perp(x,v) \geq \min \eqref{eq:beckmannproblem}.$$
Hence, $v$ is a minimizer for Problem \eqref{eq:beckmannproblem}.

(3) Let $v \in \mathcal{M}(\overline{\Omega},\mathbb{R}^2)$ be a minimizer of Problem \eqref{eq:beckmannproblem} such that $|v|(\partial\Omega) = 0$. By Proposition \ref{prop:measurerecovery}, there exists $u \in BV(\Omega)$ admissible in Problem \eqref{eq:leastgradientproblem} and $v = R_{\frac{\pi}{2}} \nabla u$. Hence, one has 
$$ \min \eqref{eq:beckmannproblem}= 
\int_\Omega \phi^\perp(x,v) = \int_\Omega \phi(x,Du) \geq \inf \eqref{eq:leastgradientproblem}.$$
Consequently, this implies that the function $u$ is in fact a solution for the problem \eqref{eq:leastgradientproblem}. 
\end{proof}

{ This concludes the equivalence between the anisotropic least gradient problem and the anisotropic Beckmann problem for any contractible domain. Let us note that if the domain is not contractible, then the results in this section are no longer true even in the Euclidean case (see for instance \cite{DG2019}, where the authors study the least gradient problem on annulus).}

\section{Optimal transport problem with Riemannian cost}\label{sec:beckmongeequivalence}
Let $k$ be a  smooth (say \,$C^{1,1}$) positive function on $\mathbb{R}^2$. We denote by $d_k$ the Riemannian metric associated with $k$:
$$d_k(x,y):=\min\bigg\{\int_0^1 k(\gamma(t))\,|\gamma^\prime(t)|\,\mathrm{d}t\,:\,\gamma \in \mbox{Lip}([0,1],\mathbb{R}^2),\,\,\gamma(0)=x\,\,\,\,\mbox{and}\,\,\,\gamma(1)=y\bigg\}.$$\\
Let $\Omega$ be an open and geodesically convex domain in $\mathbb{R}^2$ (i.e., for any two points $x$ and $y$ in $\overline{\Omega}$, there is a unique geodesic $\gamma$ contained within $\overline{\Omega}$ that joins $x$ and $y$). Let $f^+$ and $f^-$ be two nonnegative Borel measures on $\overline{\Omega}$ such that $f^+(\overline{\Omega})=f^-(\overline{\Omega})$. Then, we consider the Kantorovich problem: 
\begin{equation} \label{Kantorovich problem}
\min\bigg\{\int_{\overline{\Omega} \times \overline{\Omega}} d_k(x,y)\,\mathrm{d}\Lambda(x,y)\,:\,\Lambda \in \mathcal{M}^+(\overline{\Omega} \times \overline{\Omega}),\,\,(\Pi_x)_{\#}\Lambda=f^+\,\,\,\mbox{and}\,\,\,(\Pi_y)_{\#}\Lambda=f^-\bigg\}.
\end{equation}
In fact, Problem \eqref{Kantorovich problem} is the relaxed version of the Monge problem with Riemannian cost:
\begin{equation} \label{Monge problem}
\min\bigg\{\int_{\overline{\Omega}} d_k(x,T(x))\,\mathrm{d} f^+(x)\,:\,\,\,T_{\#}f^+=f^-\bigg\}.
\end{equation}\\
The authors of \cite{FmC} prove existence of an optimal transport map $T$ or equivalently, an optimal transport plan $\Lambda$ which is concentrated on a map, under the assumption that $f^+ \in L^1(\Omega)$. From \cite{San2015,Vil}, the problem \eqref{Kantorovich problem} admits a dual formulation
\begin{equation} \label{eq:dual}
\sup\bigg\{\int_{\overline{\Omega}} \psi\,\mathrm{d}(f^+ - f^-)\,:\,|\psi(x) - \psi(y)| \leq d_k(x,y),\,\,\forall\,\,x,\,y \in \overline{\Omega}\bigg\}.
\end{equation}
Notice that a function $\psi$ is $1-$Lipschitz with respect to the geodesic distance $d_k$ if and only if $|\nabla \psi(x)| \leq k(x)$ for almost every $x$. So, the idea in \cite{FmC} was the following: to obtain an optimal transport map $T$ in the problem \eqref{Monge problem}, it is sufficient to start from a Kantorovich
potential $\psi$ (i.e., a maximizer of the dual problem \eqref{eq:dual}) and then, construct a transport map $T$ in such a way that $\psi$ and $T$ satisfy together the following:  
$$\psi(x) - \psi(T(x))= d_k(x,T(x)),\,\,\mbox{for}\,\,f^+\,\,\mbox{a.e.}\,\,x \in \overline{\Omega}.$$
Anyway, we see that if
$\Lambda$ is an optimal transport plan in Problem \eqref{Kantorovich problem} and if $\psi$ is a Kantorovich potential in Problem \eqref{eq:dual}, then from the duality $\min\eqref{Kantorovich problem}=\sup\eqref{eq:dual}$, we infer that
$$\psi(x) -\psi(y)=d_k(x,y),\,\,\mbox{for all}\,\,(x,y) \in \mbox{spt}(\Lambda).$$
We call any maximal geodesic $\gamma_{x,y}$ between $x$ and $y$ that satisfies the equality $\psi(x) - \psi(y) = d_k(x,y)$ a transport ray. This means that any optimal transport plan $\Lambda$ moves the mass along the transport rays. An important fact is that two different transport rays cannot intersect at an interior point of one of them; to see that, assume $\gamma^+:=\gamma_{x^+,y^+}$ and $\gamma^-:=\gamma_{x^-,y^-}$ are two different transport rays. Without loss of generality, assume that they intersect at $z$, an interior point of $\gamma^-$. The point $z$ either is an interior point of $\gamma^+$ or it is one of its endpoints (by symmetry,
we assume it to be $x^+$). We have  
$$\psi(z) - \psi(y^+) = d_k(z,y^+)\,\,\,\,\mbox{and}\,\,\,\,\psi(x^-) - \psi(z) = d_k(x^-,z).$$
Then,
$$\psi(x^-) - \psi(y^+) =  d_k(x^-,z)+ d_k(z,y^+).$$
Yet, $\psi$ is $1-$Lipschitz with respect to $d_k$. Hence, $z$ belongs to a geodesic $\gamma$ from $x^-$ to $y^+$, which is a concatenation of fragments of $\gamma^+$ and $\gamma^-$.
 But, this is a contradiction as there is a unique minimizing geodesic starting at $x^-$ with initial velocity $\gamma^\prime(0)$ (thanks to the $C^{1,1}$ regularity of the Riemannian metric $k$; see \cite[Chapter 8]{CanSin} or \cite[Chapter 7]{DweikThesis}).
 
In optimal transport theory it is classical to associate with any optimal transport plan $\Lambda$ a
nonnegative measure $\sigma_\Lambda$ on $\overline{\Omega}$, called the {\it{transport density}}, which represents the amount of transport
taking place in each region of $\Omega$. This measure $\sigma_\Lambda$ is defined as follows
\begin{equation} \label{transport density definition}
<\sigma_\Lambda,\phi>:=\int_{\overline{\Omega} \times \overline{\Omega}} \int_{0}^1 \phi(\gamma_{x,y}(t)) \, k(\gamma_{x,y}(t)) \,  |\gamma^\prime_{x,y}(t)| \,\mathrm{d}t\,\mathrm{d}\Lambda (x,y),\,\,\,\,\mbox{for all}\,\,\,\,\phi \in C(\overline{\Omega}),
\end{equation}
where $\gamma_{x,y}$ is the unique geodesic between $x$ and $y$. We note that the measure $\sigma_\Lambda$ is well defined thanks to the fact that $\Omega$ is geodesically convex. Denote by $\mathcal{H}^1_k$ the weighted Hausdorff measure, i.e. $\mathcal{H}^1_k(\gamma):=\int_0^1 k(\gamma(t))\,|\gamma^\prime(t)|\,\mathrm{d}t$ for a Lipschitz curve $\gamma$. From \eqref{transport density definition}, we can see easily that
\begin{equation} \label{transport density definition 1}
\sigma_\Lambda(A):=\int_{\overline{\Omega} \times \overline{\Omega}} \mathcal{H}^1_k(\gamma_{x,y} \cap A)\,\mathrm{d}\Lambda (x,y),\,\,\,\,\mbox{for all Borel set}\,\,A \subset \overline{\Omega}.
\end{equation}
On the other hand, we define a vector measure $v_\Lambda$, which is the vector version of $\sigma_\Lambda$, as follows
\begin{equation} \label{Beckmann minimizer}
<v_\Lambda,\xi>:=\int_{\overline{\Omega} \times \overline{\Omega}} \int_{0}^1 
\xi(\gamma_{x,y}(t)) \cdot
k(\gamma_{x,y}(t))\,  \gamma^\prime_{x,y}(t) \,\mathrm{d}t\,\mathrm{d}\Lambda (x,y),\,\,\,\,\mbox{for all}\,\,\,\,\xi \in C(\overline{\Omega},\mathbb{R}^2).
\end{equation}
Let $\psi$ be a Kantorovich potential in the dual problem \eqref{eq:dual}. Thanks to \cite[Lemma 10]{FmC}, $\psi$ is differentiable at any interior point of a transport ray  
 $\gamma_{x,y}$ 
 and, we have 
$$\gamma_{x,y}^\prime(t)=-k^{-1}(\gamma_{x,y}(t))\, d_k(x,y)\,\frac{\nabla \psi(\gamma_{x,y}(t))}{|\nabla \psi (\gamma_{x,y}(t))|}.$$
This implies that
$$v_\Lambda=-\sigma_\Lambda\,\frac{\nabla \psi}{|\nabla \psi|}.$$ \\
In particular, we have $|v_\Lambda|=\sigma_\Lambda$. From \eqref{transport density definition} and the optimality of the transport plan $\Lambda$, we infer that
$$\int_{\overline{\Omega}} |v_\Lambda|=\sigma_\Lambda(\overline{\Omega})=\int_{\overline{\Omega}\times\overline{\Omega}} d_k(x,y)\,\mathrm{d}\Lambda(x,y)=\min\eqref{Kantorovich problem}.$$
 Fix $\phi \in C^1(\overline{\Omega})$. Taking \,$\xi=k^{-1} \nabla \phi \in C(\overline{\Omega}, \mathbb{R}^2)$ as a test function in the definition \eqref{Beckmann minimizer} of $v_\Lambda$, then we have
\begin{eqnarray*} 
<v_\Lambda, k^{-1} \nabla \phi> &=&\int_{\overline{\Omega} \times \overline{\Omega}} \int_{0}^1  \nabla \phi(\gamma_{x,y}(t)) \cdot \gamma^\prime_{x,y}(t) \,\mathrm{d}t\,\mathrm{d}\Lambda (x,y)\\ 
&=&\int_{\overline{\Omega} \times \overline{\Omega}} \int_{0}^1  \left[\frac{d}{d t} \phi(\gamma_{x,y}(t))\right]\,\mathrm{d}t\,\mathrm{d}\Lambda (x,y)\\
&=&\int_{\overline{\Omega} \times \overline{\Omega}}  \left[\phi(y) - \phi(x)\right]\,\mathrm{d}\Lambda (x,y)
\\
&=&\int_{\overline{\Omega}}  \phi\,\mathrm{d}(f^- - f^+),
\end{eqnarray*}\\
This implies that $\mathrm{div} [k^{-1} v_\Lambda]=f^+ - f^-$. Now, let us consider the weighted Beckmann problem
\begin{equation} \label{weighted beckmann problem}
\min \bigg\{ \int_{\overline{\Omega}} k|v| \,:\, v \in \mathcal{M}(\overline{\Omega}, \mathbb{R}^2),\,\, \mathrm{div} \, v = f \bigg\}.
\end{equation}
 Similarly to the Euclidean case, on geodesically convex domains we have an equivalence between the Beckmann problem and the optimal transport problem. By the reasoning above, we infer that
$$
\min \eqref{weighted beckmann problem} \,\,\leq  \int_{\overline{\Omega}} k |k^{-1} v_\Lambda|\,\, =  \int_{\overline{\Omega}} |v_\Lambda|\,\,=\,\, \min\eqref{Kantorovich problem}.$$
Yet, it is easy to see that the reverse inequality also holds, since for any vector measure $v$ such that $\mathrm{div} (v)=f$ and any smooth function $\psi$ such that $|\nabla \psi| \leq k$, one has 
$$\int_{\overline{\Omega}} \psi\,\mathrm{d}(f^+-f^-)=-\int_{\overline{\Omega}} \nabla \psi \cdot \mathrm{d} v \leq \int_{\overline{\Omega}} k \,\mathrm{d} |v|.$$
Hence, we get that
$$
\min \eqref{weighted beckmann problem} \,\,=\,\, \min\eqref{Kantorovich problem}.$$\\
Moreover, we infer that \,$k^{-1}v_\Lambda$ solves the Beckmann problem \eqref{weighted beckmann problem}. In fact, one can prove (see \cite[Proposition 2.3]{DS} in the Euclidean case) that every optimal flow $v$ of Problem \eqref{weighted beckmann problem} is of the form $v=k^{-1}v_\Lambda$, for some optimal transport plan $\Lambda$ of Problem \eqref{Kantorovich problem}.  Let us note that in \cite{Prat} the author proved the equivalence between the Beckmann problem and the optimal transport problem in a more general setting, but then the definition of the transport density has to be altered to allow for splitting of the transport along multiple transport rays. Since in our regularity results we rely on the uniqueness of the geodesic between any two given points, let us present a sketch of the proof of an analogue of \cite[Proposition 2.3]{DS} in the Riemannian case.

Let $\mathcal{C}$ be the set of absolutely continuous curves $\gamma:[0,1] \mapsto \overline{\Omega}$. We call a traffic plan any nonnegative measure $\eta$ on $\mathcal{C}$ such that $(e_0)_ {\#}\eta=f^+$ and $(e_1)_ {\#}\eta=f^-$, where for all $t \in [0,1]$ we denote $e_t(\gamma):=\gamma(t)$. We define the traffic flow $v_\eta$ and the traffic intensity $i_\eta$ as follows:  
\begin{equation*} 
<v_\eta,\xi>:=\int_{\mathcal{C}} \int_{0}^1  \xi(\gamma(t)) \cdot k(\gamma(t))\, \gamma^\prime(t) \,\mathrm{d}t\,\mathrm{d}\eta (\gamma),\,\,\,\,\mbox{for all}\,\,\,\,\xi \in C(\overline{\Omega},\mathbb{R}^2),
\end{equation*} 
and
\begin{equation*}
<i_\eta,\phi>:=\int_{\mathcal{C}} \int_{0}^1  \phi(\gamma(t))\,k(\gamma(t))\, |\gamma^\prime(t)| \,\mathrm{d}t\,\mathrm{d}\eta(\gamma),\,\,\,\,\mbox{for all}\,\,\,\,\phi \in C(\overline{\Omega}).
\end{equation*} \\
We note that $\mathrm{div} [k^{-1}\,v_\eta]=f^+ - f^-$; in particular, it is admissible in Problem \eqref{weighted beckmann problem}. Let $v$ be an optimal flow for Problem \eqref{weighted beckmann problem}. Similarly to \cite[Lemma 2.2]{DS}, one can show that there is a traffic plan $\eta$ such that 
$$\int_{\overline{\Omega}} \mathrm{d}|kv-  v_\eta| + \int_{\overline{\Omega}} \,\mathrm{d}i_\eta=\int_{\overline{\Omega}} k\,\mathrm{d}|v|.$$
Note that $|v_\eta| \leq i_\eta$. Since $k^{-1} v_\eta$ is admissible in Problem \eqref{weighted beckmann problem} while $v$ minimizes Problem \eqref{weighted beckmann problem}, then we get
\begin{equation}\label{eq:boundonieta}
\int_{\overline{\Omega}} \,\mathrm{d}i_\eta \geq \int_{\overline{\Omega}} k\,\mathrm{d}|v|.    
\end{equation}
Moreover, we have
$$\int_{\overline{\Omega}} k\,\mathrm{d}|v| \leq \int_{\overline{\Omega}} \mathrm{d}|k v-v_\eta| 
+ \int_{\overline{\Omega}} \,\mathrm{d}|v_\eta|  \leq\int_{\overline{\Omega}} \mathrm{d}|kv-  v_\eta|
+ \int_{\overline{\Omega}} \,\mathrm{d}i_\eta=\int_{\overline{\Omega}} k\,\mathrm{d}|v|.$$
 Hence, all the inequalities above are in fact equalities. But then, equation \eqref{eq:boundonieta} implies that
$$ v=k^{-1}v_\eta\qquad \mbox{and}\qquad |v_\eta|=i_\eta. $$
Set $\Lambda_\eta:=(e_0,e_1)_{\#}\eta$. Then, it is clear that $\Lambda_\eta$ is a transport plan between $f^+$ and $f^-$. In addition, we have
\begin{equation*} 
\int_{\overline{\Omega}}\mathrm{d}|v_\eta|=\int_{\mathcal{C}} \int_{0}^1 k(\gamma(t))\,| \gamma^\prime(t)| \,\mathrm{d}t\,\mathrm{d}\eta (\gamma) \geq \int_{\mathcal{C}} d_k(\gamma(0),\gamma(1))\,\mathrm{d}\eta (\gamma)=\int_{\overline{\Omega}}d_k(x,y)\,\mathrm{d}\Lambda_\eta.
\end{equation*}
Yet, $\min\eqref{weighted beckmann problem}=\min\eqref{Kantorovich problem}$ and $k^{-1}v_\eta$ minimizes Problem \eqref{weighted beckmann problem}. Hence, the inequality above is an equality, which implies that for $\eta-$a.e. $\gamma \in \mathcal{C}$, $d_k(\gamma(0),\gamma(1))=\int_{0}^1 k(\gamma(t))| \gamma^\prime(t)| \,\mathrm{d}t$ and so, $\gamma$ is the unique geodesic between $\gamma(0)$ and $\gamma(1)$. Consequently, $v=k^{-1}v_\eta=k^{-1}v_{\Lambda_\eta}$ and $\Lambda_\eta$ is an optimal transport plan for Problem \eqref{Kantorovich problem}.\\



The $L^p$ summability of the transport density $\sigma_\Lambda$ was already considered in many papers \cite{San,DePas1,DePas2,DePas3}. However, in all these works, the authors just consider the Euclidean case (i.e., $k \equiv 1$). More precisely, they show that for all $p \in [1,\infty]$ the transport density $\sigma_\Lambda$ belongs to $L^p(\Omega)$ as soon as $f^\pm \in L^p(\Omega)$. Moreover, the authors of \cite{DS} have already considered the case where the source and target measures are singular ($f^\pm$ are two measures concentrated on $\partial\Omega$): they show (again, under the assumption that $k \equiv 1$) that if $f^\pm \in L^p(\partial\Omega)$, then the transport density $\sigma_\Lambda$ between them is in $L^p(\Omega)$ provided that $p \leq 2$ and $\Omega$ is uniformly convex. In this paper, we will study the $L^p$ summability of the transport density $\sigma_\Lambda$ in the case where $f^\pm$ are concentrated on the boundary $\partial\Omega$ and $k$ is a positive $C^{1,1}$ Riemannian metric. \\

Let us come back to the question of existence and uniqueness of an optimal transport map $T$ for the Monge problem \eqref{Monge problem}. We note that this map $T$ (when it exists) is not necessarily unique and, the Kantorovich problem \eqref{Kantorovich problem} may have many different solutions as  well. Under the assumption that $f^+ \in L^1(\Omega)$, one can pick a special optimal transport plan $\Lambda$ which will be induced by a transport map $T$, and this map turns out to be a solution for the Monge problem \eqref{Monge problem}. Yet, in this paper, we consider the case where $f^+$ and $f^-$ are two nonnegative measures which are concentrated on $\partial\Omega$ and so, it is not clear if the Monge problem \eqref{Monge problem} reaches a minimum or not. Before studying existence and uniqueness of minimizers for Problems \eqref{Kantorovich problem} \& \eqref{Monge problem}, we need to introduce the following:
\begin{definition}
We say that $\Omega$ is geodesically strictly convex if given any two points in $\overline{\Omega}$, the unique minimizing geodesic between them lies in the interior of \,$\Omega$, possibly except for its endpoints.
\end{definition}
Then, we have the following existence and uniqueness result (the proof is similar to the one given in \cite[Proposition 2.5]{DS}, but for the sake of completeness we will introduce here its adaptation to the Riemannian case):
\begin{proposition}\label{prop:transportplan}
Assume that $\Omega$ is geodesically strictly convex.  
Then, there is a unique optimal transport plan
$\Lambda$, between $f^+$ and $f^-$, which will be induced by a transport map $T$, provided that $f^+$ is non-atomic. This map $T$ turns out to be the unique optimal transport map in the Monge problem \eqref{Monge problem}. 
\end{proposition}
\begin{proof}
Let $\Lambda$ be an optimal transport plan between $f^+$ and $f^-$. Let us denote by $D$ the set of points whose belong to different transport rays. Since two transport rays cannot intersect at an interior point of either of them, we have that $D \subset \partial\Omega$. Fix $x \in D$ and let $\gamma^\pm_x$ be
two different transport rays starting at
$x$. Let $E_x$ be the region delimited by $\gamma_x^+$, $\gamma_x^-$ and $\partial\Omega$.
Then, we easily see that these sets $\{E_x\}_{x \in D}$ must be essentially disjoint with $\mathcal{L}^2(E_x) >0$, for every $x \in D$. This implies that the set $D$ is at most countable. Yet, $f^+$ is non-atomic and so, $f^+(D)=0$. On the other hand, we have that for $f^+-$almost every
$x \,\notin D$,
there is a unique transport ray $\gamma_x$ starting at $x$ and, thanks to the fact that $\Omega$ is  geodesically strictly convex, this geodesic $\gamma_x$ intersects \,$\mbox{spt}(f^-)$
at exactly one point $T(x)$. This implies that $\Lambda=(Id,T)_{\#}f^+$. Yet, this is sufficient to infer that $\Lambda$ is the unique optimal transport plan in the Kantorovich problem \eqref{Kantorovich problem} since, if $\Lambda^\prime=(Id,T^\prime)_{\#}f^+$ is another optimal transport plan then \,$\Lambda^{\prime\prime}=(\Lambda + \Lambda^\prime)/2$\, is also optimal and so, $\Lambda^{\prime\prime}$ must be induced by a transport map $T^{\prime\prime}$. But, this yields to a contradiction as soon as $T\neq T^\prime$, since it is not possible to have the following equality for all $\varphi \in C(\overline{\Omega} \times \overline{\Omega})$:
$$\int_{\overline{\Omega}}\varphi(x,T^{\prime\prime}(x))\,\mathrm{d}f^+(x)=\frac{1}{2}\bigg[\int_{\overline{\Omega}}\varphi(x,T(x))\,\mathrm{d}f^+(x)+\int_{\overline{\Omega}}\varphi(x,T^{\prime}(x))\,\mathrm{d}f^+(x)\bigg]. $$
\end{proof}


Fix $\tau \in [0,1]$. Then, we define the partial transport densities $\sigma_\Lambda^+$ and $\sigma_\Lambda^-$ as follows:
\begin{equation} \label{partial transport density definition 1}
<\sigma_\Lambda^+,\phi>:=\int_{\overline{\Omega} \times \overline{\Omega}} \int_{0}^{\tau} \phi(\gamma_{x,y}(t))\,k(\gamma_{x,y}(t))\, |\gamma^\prime_{x,y}(t)| \,\mathrm{d}t\,\mathrm{d}\Lambda (x,y),\,\,\,\,\mbox{for all}\,\,\,\,\phi \in C(\overline{\Omega}),
\end{equation}
and 
\begin{equation} \label{partial transport density definition 2}
<\sigma_\Lambda^-,\phi>:=\int_{\overline{\Omega} \times \overline{\Omega}} \int_{\tau}^1 \phi(\gamma_{x,y}(t))\,k(\gamma_{x,y}(t))\, |\gamma^\prime_{x,y}(t)| \,\mathrm{d}t\,\mathrm{d}\Lambda (x,y),\,\,\,\,\mbox{for all}\,\,\,\,\phi \in C(\overline{\Omega}). 
\end{equation}
It is clear that $\sigma_\Lambda=\sigma_\Lambda^+ + \sigma_\Lambda^-$. To prove $L^p$ estimates on the transport density $\sigma_\Lambda$, the idea in the next sections will be to prove $L^p$ summability on $\sigma_\Lambda^\pm$. \\

In what follows, we will also need the following stability result.
\begin{proposition} \label{stability}
Assume that $f^\pm \in \mathcal{M}^+(\partial\Omega)$ and let $(f_n^-)_n \subset \mathcal{M}^+(\partial\Omega)$ be such that $f_n^- \rightharpoonup f^-$. Let $\Lambda_n$ be an optimal transport plan between $f^+$ and $f_n^-$. Then, $\Lambda_n \rightharpoonup \Lambda$, where $\Lambda$ is an optimal transport plan between $f^+$ and $f^-$. Moreover, $\sigma_{\Lambda_n}^\pm \rightharpoonup \sigma_\Lambda^\pm$. Finally, if $\Lambda_n:=(I,T_n)_{\#}f^+$ and $\Lambda:=(I,T)_{\#}f^+$, then (up to taking a subsequence) for $f^+$-a.e. $x \in \partial\Omega$, we have $T_n(x) \rightarrow T(x)$.
\end{proposition}

\begin{proof}
Let $\psi_n$ be a Kantorovich potential between $f^+$ and $f_n^-$. Then, from the duality $\min\eqref{Kantorovich problem}=\sup\eqref{eq:dual}$, we have 
$$\int_{\overline{\Omega} \times \overline{\Omega}} d_k(x,y)\,\mathrm{d}\Lambda_n(x,y)=\int_{\overline{\Omega}} \psi_n\,\mathrm{d}(f^+ - f_n^-).$$ 
Yet, it is clear that, up to a subsequence, $\psi_n \rightarrow \psi$ uniformly. So, passing to the limit when $n \to \infty$, we get 
$$\int_{\overline{\Omega} \times \overline{\Omega}} d_k(x,y)\,\mathrm{d}\Lambda(x,y)=\int_{\overline{\Omega}} \psi\,\mathrm{d}(f^+ - f^-).$$\\
This implies that $\Lambda$ is an optimal transport plan between $f^+$ and $f^-$, while $\psi$ is a Kantorovich potential between them. The second statement follows directly from the definitions \eqref{partial transport density definition 1} \& \eqref{partial transport density definition 2} of $\sigma_\Lambda^\pm$. The last statement follows directly from the fact that
$$\int_{\overline{\Omega}} \xi(x, T_n(x))\,\mathrm{d}f^+(x) \rightarrow \int_{\overline{\Omega}} \xi(x,T(x))\,\mathrm{d}f^+(x),\,\,\,\mbox{for all}\,\,\xi \in C(\overline{\Omega} \times \overline{\Omega}).$$
Take $\xi(x,y)=\xi^+(x) \cdot y$\, and \,$\xi(x,y)=|y|^2$ respectively, we infer that $T_n \rightharpoonup T$ in $L^2$ and $||T_n||_{L^2} \rightarrow ||T||_{L^2}$ and so,  $T_n \rightarrow T$ in $L^2$. In particular, it converges $f^+$-a.e. on a subsequence.
\end{proof}

{ We will conclude this section by commenting on the relationship between geodesic convexity of the domain and the barrier condition (see Definition \ref{Barrier Condition definition}), which is the assumption used for existence of solutions in the least gradient problem in \cite{JMN}. It turns out that the assumption of geodesic strict convexity of $\Omega$, required to use optimal transport techniques, implies the barrier condition.
 
\begin{proposition}
 Suppose that $\Omega$ is geodesically strictly convex. Then, it satisfies the barrier condition.
\end{proposition}
\begin{proof}
Let $x_0 \in \partial\Omega$. Because $\Omega$ has Lipschitz boundary, we may choose $r > 0$ small enough so that $\partial\Omega \cap \partial B(x_0,r)$ consists of two points; we denote them by $x_1$ and $x_2$. Denote by $\gamma_{x_1 x_2}$ the geodesic from $x_1$ to $x_2$; because $\Omega$ is geodesically strictly convex, we have $\gamma_{x_1 x_2} \subset \Omega$.

Now, we consider the problem \eqref{eq:barriercondition} in the definition of the barrier condition. Using the direct method of the calculus of variations, we easily see that it admits a solution $V$. The points $x_1, x_2$ separate $\partial\Omega$ into two arcs: $\Gamma_1$ is the arc which contains $x_0$ and $\Gamma_2$ is the arc which does not contain $x_0$. On the other hand, notice that $\gamma_{x_1 x_2}$ separates $\Omega$ into two regions; let us denote by $V'$ the region whose closure contains $\Gamma_2$. Hence, we may take $V'' = V' \cup (\Omega \backslash B(x_0,r))$ as a competitor in \eqref{eq:barriercondition} and see that $V = V''$; this happens because $V'$ minimizes perimeter in a larger class, namely it is a solution of
\begin{equation}\label{eq:similarproblem}
\inf \bigg\{ P_\phi(W, \mathbb{R}^2): \quad W \subset \Omega, \quad T\chi_W \geq \chi_{\Gamma_2} \bigg\},
\end{equation}
so as $V$ is admissible in \eqref{eq:barriercondition}, then it is also admissible in \eqref{eq:similarproblem}; in particular, $P_\phi(V', \mathbb{R}^2) \leq P_\phi(V, \mathbb{R}^2)$. This inequality still holds after taking a union with $\Omega \backslash B(x_0, r)$, see for instance \cite{Mag} (or simply notice that $\partial V''$ is a Lipschitz curve which arises from replacing a part of $\partial V'$ in case when it leaves $B(x_0,r)$ with a part of $\partial B(x_0,r)$). In particular, $\partial V$ consists of parts of $\gamma_{x_1,x_2}$ and parts of $\partial B(x_0,r)$, so $\partial V \cap \partial\Omega \cap B(x_0,r) = \emptyset$ and the barrier condition is satisfied.
\end{proof}

On the other hand, the barrier condition is only a local property near $\partial\Omega$ and does not imply geodesic convexity of $\Omega$, because the geodesic between any two given points may fail to be unique. This can be seen by considering the examples in \cite{LMSS} (in this regard, they may be modified for a smooth weight). However, the barrier condition means that the boundary is not locally area-minimizing with respect to internal variations, so once we know that the geodesic between any two given points in $\overline{\Omega}$ is unique, it means that the domain is strictly geodesically convex.
}

{

\section{Structure of solutions}\label{sec:structure}



In this section, we extend the relationship between the weighted least gradient problem \eqref{eq:weightedleastgradientproblem} and the Kantorovich problem with Riemannian cost \eqref{Kantorovich problem} to their respective dual problems. Namely, we study the relationship between the maximization problem (see \cite{Mor})
\begin{equation}\label{eq:dualtolgp}
\sup \bigg\{ \int_{\partial\Omega} [\mathbf{z},n] \, g \, \mathrm{d}\mathcal{H}^{1}: \mathbf{z} \in \mathcal{Z} \bigg\},
\end{equation}
where $g \in BV(\partial\Omega)$ and
\begin{equation*}
\mathcal{Z} = \bigg\{ \mathbf{z} \in L^\infty(\Omega,\mathbb{R}^2), \quad \mathrm{div} \, \mathbf{z} = 0, \quad |\mathbf{z}(x)| \leq k(x) \mbox{  for a.e. } x \in \Omega \bigg\},
\end{equation*}
with the maximization problem \eqref{eq:dual}
\begin{equation*}
\sup\bigg\{\int_{\overline{\Omega}} \psi \,\mathrm{d}(f^+ - f^-)\,:\,|\psi(x) - \psi(y)| \leq d_k(x,y),\,\,\forall\,\,x,\,y \in \overline{\Omega} \bigg\}.
\end{equation*}
Here, the normal trace $[\mathbf{z},n]$ is understood in the weak sense (see Section \ref{sec:anisotropicbv}). Using a standard reasoning in duality theory, one can see that both problems admit solutions; see \cite{Mor} for problem \eqref{eq:dualtolgp} and \cite{San2015,Vil} for problem \eqref{eq:dual}. Since the infimal values in the primal problems are equal, the supremal values in the dual problems also coincide. We now show that these problems are equivalent in the sense that from a solution of one problem we may recover a solution of the other problem. In the Euclidean case, such a result was shown in \cite{Gor2021Appl}. We use a similar technique in the weighted case; the main differences are that instead of uniform bounds on $\mathbf{z}$ we now have pointwise bounds and the domain is no longer convex, so we need to rely on Proposition \ref{prop:L1recovery} (valid for any contractible domain) to recover the Kantorovich potential.

\begin{theorem}\label{thm:dualproblems}
 Suppose that $\Omega$ is geodesically convex. Then, the problems \eqref{eq:dual} and \eqref{eq:dualtolgp} are equivalent in the following sense: \\
(1) Their supremal values coincide, i.e. $\sup \eqref{eq:dual} = \sup \eqref{eq:dualtolgp}$. \\
(2) Given a maximizer $\psi \in \mathrm{Lip}(\overline{\Omega})$ of \eqref{eq:dual}, we can construct a maximizer $\mathbf{z} \in L^\infty(\Omega, \mathbb{R}^2)$ of \eqref{eq:dualtolgp}; moveover, $\mathbf{z} = R_{-\frac{\pi}{2}} \nabla \psi$ in $\Omega$. \\
(3) Given a maximizer $\mathbf{z} \in L^\infty(\Omega,\mathbb{R}^2)$ of \eqref{eq:dualtolgp}, we may construct a maximizer $\psi \in \mathrm{Lip}(\overline{\Omega})$ of \eqref{eq:dual}and, we have $\mathbf{z} = R_{-\frac{\pi}{2}} \nabla \psi$ in $\Omega$.
\end{theorem}

Notice that the direction of the rotation is opposite to the direction of rotation in Theorem \ref{thm:lgpbeckmannequivalence}. We need geodesic convexity of $\Omega$, because in the computation of supremal values we need to pass through the primal problems, and the equivalence for primal problems holds on geodesically convex domains.

\begin{proof}
(1) This follows immediately from the equivalence between the weighted least gradient problem, weighted Beckmann problem, and the Monge-Kantorovich problem. We have
\begin{equation*}
\sup \eqref{eq:dualtolgp} = \inf \eqref{eq:weightedleastgradientproblem} = \inf \eqref{Kantorovich problem} = \sup \eqref{eq:dual}.
\end{equation*}
(2) Suppose that $\psi \in \mathrm{Lip}(\overline{\Omega})$ is a maximizer in Problem \eqref{eq:dual}. Take $\mathbf{z} = R_{-\frac{\pi}{2}} \nabla \psi \in L^\infty(\Omega, \mathbb{R}^2)$. Then, $\mathbf{z}$ is admissible in \eqref{eq:dualtolgp}, since $\mathrm{div} (\mathbf{z}) = 0$ as distributions and $|\mathbf{z}(x)| = | \nabla \psi(x) | \leq k(x)$ a.e. in $\Omega$. Since $\psi$ is Lipschitz and $g \in BV(\partial\Omega)$, we have $\psi g \in BV(\partial\Omega)$ and a mass balance condition holds:
\begin{equation}\label{eq:lipbvmassbalance}
0 = \int_{\partial\Omega} \mathrm{d} [\partial_\tau(\psi g)] = \int_{\partial\Omega} \psi \, \mathrm{d}(\partial_\tau g) + \int_{\partial\Omega} (\partial_\tau \psi) \, g \, \mathrm{d}\mathcal{H}^1.
\end{equation}
Now, $\mathbf{z} = R_{-\frac{\pi}{2}} \nabla \psi = - R_{\frac{\pi}{2}} \nabla \psi$, so by Proposition \ref{prop:L1recovery} we have $[\mathbf{z},n] = - \partial_\tau \psi$. By \eqref{eq:lipbvmassbalance}, 
\begin{equation*}
\sup \eqref{eq:dualtolgp} = \sup \eqref{eq:dual} = \int_{\partial\Omega} \psi \, \mathrm{d}f = \int_{\partial\Omega} \psi \, d(\partial_\tau g) =  \int_{\partial\Omega} (- \partial_\tau \psi) \, g \, \mathrm{d}\mathcal{H}^1 =
\int_{\partial\Omega} [\mathbf{z},n] \, g \, \mathrm{d}\mathcal{H}^1.
\end{equation*}
Hence, $\mathbf{z}$ is a maximizer for Problem \eqref{eq:dualtolgp}. \\

(3) Suppose that $\mathbf{z} \in L^\infty(\Omega, \mathbb{R}^2)$ is a maximizer in \eqref{eq:dualtolgp}. 
By Proposition \ref{prop:L1recovery} there exists $\overline{\psi} \in W^{1,1}(\Omega)$ such that $\mathbf{z} = R_{\frac{\pi}{2}} \nabla \overline{\psi}$. Since $\mathbf{z} \in L^\infty(\Omega,\mathbb{R}^2)$ with $| \mathbf{z}(x) | \leq k(x)$ a.e. in $\Omega$, we also have $\nabla \overline{\psi} \in L^\infty(\Omega,\mathbb{R}^2)$ with $| \nabla \overline{\psi}(x)| \leq k(x)$ a.e. in $\Omega$. 
Hence, $\overline{\psi}$ is $1$-Lipschitz with respect to the geodesic distance $d_k$, so it is admissible in Problem \eqref{eq:dual}. Moreover, Proposition \ref{prop:L1recovery} implies that $[\mathbf{z},n] = \partial_\tau \overline{\psi}$. Again using equation \eqref{eq:lipbvmassbalance}, we get
\begin{equation*}
\sup \eqref{eq:dual} = \sup \eqref{eq:dualtolgp} = \int_{\partial\Omega} [\mathbf{z},n] \, g \, \mathrm{d}\mathcal{H}^1 = \int_{\partial\Omega} (\partial_\tau \overline{\psi}) \, g \, \mathrm{d}\mathcal{H}^1 = - \int_{\partial\Omega} \overline{\psi} \, d(\partial_\tau g) = \int_{\partial\Omega} (-\overline{\psi}) \, \mathrm{d}f.
\end{equation*}
Hence, $\psi = -\overline{\psi}$ is a maximizer for Problem \eqref{eq:dual}. In particular, $\mathbf{z} = R_{\frac{\pi}{2}} \nabla \overline{\psi} = R_{-\frac{\pi}{2}} \nabla \psi$. 
\end{proof}

Hence, we may express the solution to the dual of the weighted least gradient problem \eqref{eq:dualtolgp} via a Kantorovich potential of the corresponding optimal transport problem and vice versa. In general, we cannot expect solutions to any of these problems to be unique. However, since the Kantorovich potentials are differentiable in the interiors of the transport rays and their gradient is uniquely defined, their structure is somewhat prescribed by the boundary data. We will now study some consequences of this result for the structure of solutions to the weighted least gradient problem. For the remainder of this section, for $u \in BV(\Omega)$, we denote $E_t := \{ u \geq t \}$.

\begin{corollary}
 Suppose that $\Omega$ is geodesically convex. Let $u$ be a solution to the weighted least gradient problem with boundary data $g \in BV(\partial\Omega)$. Then, for all $t \in \mathbb{R}$, every connected component of $\partial E_t$ is a transport ray.
\end{corollary}

This result is a precise statement of the informal observation that boundaries of superlevel sets of a solution to the least gradient problem are transport rays, which was made in the Euclidean case in \cite{DS}.


\begin{proof}
By Theorem \ref{thm:anisotropicbgg}, characteristic functions of superlevel sets $\chi_{E_t}$ are functions of least gradient, for all $t \in \mathbb{R}$. By Theorem \ref{thm:mrlcharacterisation}, there exists a vector field $\mathbf{z} \in L^\infty(\Omega, \mathbb{R}^2)$ with $|z| \leq k$ and zero divergence such that $(\mathbf{z},Du) = k |Du|$; in particular, $\mathbf{z}$ maximizes \eqref{eq:dualtolgp} since $u$ is a solution for Problem \eqref{eq:weightedleastgradientproblem}. Notice that by the co-area formula 
(see \cite[Proposition 2.7]{Anz} 
) the vector field $\mathbf{z}$ satisfies $(\mathbf{z},D\chi_{E_t}) = k |D\chi_{E_t}|$; technically, after the application of the co-area formula this result holds for almost all $t$, but since the vector field $\mathbf{z}$ is fixed we may approximate $\chi_{E_t}$ with $\chi_{E_{t_n}}$ in $L^1$ norm and obtain the result for all $t$. Now, since the distributional derivative of $\chi_{E_t}$ is concentrated on its jump set ($|D \chi_{E_t}|=\mathcal{H}^1_{|\partial E_t}$), then by \cite[Theorem 3.3]{CdC}, we have $[\mathbf{z},n^{E_t}] = k$\, $\mathcal{H}^{1}$-a.e. on $\partial E_t$; if we understand by $\mathbf{z}$ its precise representative, it means that $\mathbf{z} =k\cdot n^{E_t}$ $\mathcal{H}^{1}$-a.e. on $\partial E_t$. Hence, by Theorem \ref{thm:dualproblems}, there is a Kantorovich potential $\psi$ such that 
$\nabla \psi =k\, (n^{E_t})^\perp$ $\mathcal{H}^{1}$-a.e. on $\partial E_t$, so 
every connected component of $\partial E_t$ 
is in fact a transport ray.
\end{proof}

For the remainder of this section, let us introduce the following notation. We will work under the assumption that $\Omega$ is geodesically convex. For two points $x,\,y \in \overline{\Omega}$, denote by $\gamma_{x,y}$ the unique geodesic between these two points. Given $g \in BV(\partial\Omega)$, we denote by $D$ the set of discontinuity points of $g$, which is at most a countable set. We also denote by $\mathcal{D}$ the union of all geodesics between points in $D$. Since there is exactly one geodesic between two given points in $D$, the set $\mathcal{D}$ is a union of at most countably many geodesics and it has Hausdorff dimension one (possibly with infinite measure).

The following several lemmas are steps in the proof of Theorem \ref{thm:structure}, but are also of independent interest as results on the boundary behaviour and the jump set of a solution. Let us note that under the stronger assumption that $\Omega$ is geodesically strictly convex, it satisfies the barrier condition, and the first two results follow from the analysis in \cite{Gor2021IUMJ,Mor}. 


\begin{lemma}\label{lem:weakmaximumprinciple}
Suppose that $\Omega$ is geodesically convex. Let $u$ be a solution to the weighted least gradient problem with boundary data $g \in BV(\partial\Omega)$. Then, for all $t \in \mathbb{R}$ and every $x \in \partial\Omega$, there is at most one connected component $R_x$ of $\partial E_t$ such that $x \in R_x$.
\end{lemma}

 The proof is similar to the one given in \cite[Proposition 3.5]{Gor2018CVPDE} in the Euclidean case, but we will present here its adaptation to the Riemannian case in order to highlight the main differences and difficulties. Before we proceed, let us note that since $k \in C^{1,1}(\overline{\Omega})$, geodesic convexity of $\Omega$ implies that its boundary cannot have corners with angle larger than $\pi$: 
  the geodesics between any two points in a 
 small neighbourhood on $\partial\Omega$ of a corner point $x_0$ are uniformly $C^{1,1}$ and contained in $\overline{\Omega}$.

\begin{proof}

Suppose that there are multiple connected components of $\partial E_t$ ending at $x$. Since all of them are transport rays, they cannot intersect in $\Omega$. Therefore, any connected component $R$ of $\partial E_t$ splits $\Omega$ into two open sets, and any other connected component $R'$ of $\partial E_t$ lies entirely in one of these sets. Let $R_1$ and $R_2$ be two adjacent connected components (if it is not possible to find adjacent connected components of $\partial E_t$, because between any given two there are infinitely many, then the perimeter of $E_t$ is infinite, which contradicts Theorem \ref{thm:anisotropicbgg}). Take any $x_1 \in R_1 \cap \Omega$ and $x_2 \in R_2 \cap \Omega$. Consider the curvilinear triangle $\Delta^x_{x_1, x_2}$ with sides $\gamma_{x, x_1}$, $\gamma_{x, x_2}$ and $\gamma_{x_1, x_2}$. It does not reduce to a one-dimensional structure: by geodesic convexity of $\Omega$ and uniqueness of geodesic from $x$ with given initial velocity, $R_1$ and $R_2$ make a positive angle with each other and with the boundary $\partial\Omega$ at $x$. Therefore, the angle between $R_1$ and $R_2$ is smaller than $\pi$, so the concatenation of $\gamma_{x, x_1}$ and $\gamma_{x, x_2}$ cannot be a minimizing curve from $x_1$ to $x_2$; otherwise, this contradicts the regularity of Riemannian geodesics. But then, we can modify the set $E_t$ by adding to it (or removing from it, depending on the position of $R_1$ and $R_2$ with respect to $E_t$) the curvilinear triangle $\Delta^x_{x_1, x_2}$, without changing the trace of $\chi_{E_t}$ and reducing the perimeter of $E_t$. This again contradicts Theorem \ref{thm:anisotropicbgg}, so there can be only one connected component $R_x$ of $\partial E_t$ such that $x \in R_x$.
\end{proof}

\begin{lemma}\label{lem:boundarybehaviour}
Suppose that $\Omega$ is geodesically convex. Let $u$ be a solution to the weighted least gradient problem with boundary data $g \in BV(\partial\Omega)$. Then, for all $t \in \mathbb{R}$, we have $\partial E_t \cap \partial\Omega \subset g^{-1}(t) \cup D$.
\end{lemma}

\begin{proof}
Suppose that $x \in \partial E_t \cap (\partial \Omega \backslash D)$. Denote by $R_x$ the unique (due to Lemma \ref{lem:weakmaximumprinciple}) connected component of $\partial E_t$ passing through $x$. Then, by geodesic convexity of $\Omega$ and uniqueness of geodesic from $x$ with given initial velocity, $R_x$ makes a positive angle with $\partial\Omega$ at $x$. Yet, $g$ is continuous at $x$ and so, we may proceed in the same way as in the proof of Lemma \cite[Lemma 3.8]{GRS2017NA} (see also \cite[Lemma 3.4]{MNT}) to obtain that $g(x) = t$.
\end{proof}

\begin{lemma}\label{lem:discontinuity}
Suppose that $\Omega$ is geodesically convex. Let $u$ be a precise representative of a solution to the weighted least gradient problem with boundary data $g \in BV(\partial\Omega)$. Then, $u$ is continuous at every point in $\Omega \backslash \mathcal{D}$.
\end{lemma}

\begin{proof}
Suppose that $u$ is not continuous at $x \in \Omega$. Then, there exist $t,s \in \mathbb{R}$ such that $t < s$ and $x \in \partial E_t \cap \partial E_s$. However, there is at most one transport ray $R_x$ passing through $x$. Since the connected components of $\partial E_t$ and $\partial E_s$ are transport rays, these sets coincide on the whole connected component passing through $x$ and are equal to $R_x$. By Lemma \ref{lem:boundarybehaviour}, this is only possible if this transport ray has both ends in the discontinuity set $D$, so $x \in R_x \subset \mathcal{D}$.
\end{proof}

The following result is the main result in this section. It is a generalisation of \cite[Theorem 1.1]{Gor2018JMAA} to the weighted least gradient problem. Notice that due to the fact that we use optimal transport methods, our proof is significantly shorter and simpler.

\begin{theorem}\label{thm:structure}
Suppose that $\Omega$ is geodesically convex. Let $u,v$ be two solutions to the weighted least gradient problem with boundary data $g$. Then, $u = v$ on $\Omega \backslash (\mathcal{C} \cup \mathcal{D})$, where both $u$ and $v$ are locally constant on $\mathcal{C}$.
\end{theorem}


\begin{proof}
Consider any point $x \in \Omega \backslash \mathcal{D}$. By Lemma \ref{lem:discontinuity}, both functions $u$ and $v$ are continuous at $x$. Recall that $E_t=\{u \geq t\}$ and set $F_t=\{v \geq t\}$. There are four possibilities:

{\flushleft 1.} We have $x \in \partial E_t \cap \partial F_t$. Then, since $u$ and $v$ are continuous at $x$, we have $u(x) = v(x) = t$.

{\flushleft 2.} We have $x \in \partial E_t \cap \partial F_s$, where $t \neq s$. We argue similarly as in the proof of Lemma \ref{lem:discontinuity}. There is at most one transport ray $R_x$ passing through $x$, and because the connected components of $\partial E_t$ and $\partial F_s$ are transport rays, they need to coincide on the whole connected component passing through $x$ and so, they are equal to $R_x$. But, by Lemma \ref{lem:boundarybehaviour}, this is only possible if the transport ray has both ends in the discontinuity set $D$, in which case by uniqueness of the geodesic between any two given points we have $R_x \subset \mathcal{D}$, so $x \in \mathcal{N}$, a contradiction.

{\flushleft 3.} We have $x \in \partial E_t$, but $x \notin \partial F_s$ for any $s \in \mathbb{R}$. Then, $v$ is constant in a ball $B(x,r)$ with value $s \in \mathbb{R}$. If $t = s$, the claim is proved; suppose otherwise. Since $u$ is continuous at $x$ with $x \in \partial E_t$, there exists an interval $\mathcal{I}_\varepsilon = [t-\varepsilon,t+\varepsilon]$ such that for all $\tau \in \mathcal{I}_\varepsilon$ a connected component $\Gamma_\tau$ of $\partial E_\tau$ intersects the ball $B(x,r)$. Denote by $x_\tau$ and $y_\tau$ the endpoints of $\Gamma_\tau$ and denote by $C_s$ the connected component of $\{ v \equiv s \}$ containing $B(x,r)$. Since transport rays cannot intersect inside $\Omega$, we have $\Gamma_\tau \cap \partial C_s = \emptyset$ for all $\tau \in \mathcal{I}_\varepsilon$; then, by checking the trace of $v$, we see that on the arc $(x_{t - \varepsilon}, x_{t + \varepsilon})$ on $\partial\Omega$ we have $g \equiv s$. On the other hand, since $x \notin \mathcal{D}$, we may assume that $g$ is continuous at $x_t$ and by Lemma \ref{lem:boundarybehaviour} we have that $g(x_t) = t$, a contradiction.

{\flushleft 4.} We have $x \notin \partial E_t$ for any $t \in \mathbb{R}$ and $x \notin \partial F_s$ for any $s \in \mathbb{R}$. Then, both $u$ and $v$ are constant in a ball $B(x,r)$ (possibly with different constants), so $x \in \mathcal{C}$.
\end{proof}

}

\section{$L^p$ estimates on the transport density}\label{sec:lpestimates}

In this section, we study the $L^p$ summability of the transport density $\sigma$ between $f^+$ and $f^-$. Assume that $\Omega$ is  geodesically strictly convex. From Proposition \ref{prop:transportplan}, we recall that there is a unique optimal transport plan $\Lambda$ and so, the transport density $\sigma$ is unique, provided that $f^+$ or $f^-$ is non-atomic. To prove $L^p$ estimates on $\sigma$, the strategy will be the following: we decompose $\sigma$ into two parts $\sigma^+$ and $\sigma^-$, then we approach the target (resp. source) measure $f^-$ (resp. $f^+$) with atomic measures, and under some geometric assumptions on $\Omega$, we prove that $\sigma^+$ (resp. $\sigma^-$) is in $L^p(\Omega)$ as soon as $f^+$ (resp. $f^-$) belongs to $L^p(\partial\Omega)$ with $p\leq 2$ (in order to obtain $L^p$ estimates on $\sigma$ with $p>2$, we need to assume more regularity on $f^\pm$). We note that this $L^p$ summability of the transport density between two measures $f^+$ and $f^-$ which are in $L^p(\partial\Omega)$ was already proved in the Euclidean case (i.e., $k \equiv 1$) in \cite{DS}; more precisely, the authors of \cite{DS} proved that if $f^\pm \in L^p(\partial\Omega)$ then $\sigma \in L^p(\Omega)$ provided that $p \leq 2$ and $\Omega$ is uniformly convex. 

First, we are going to  introduce a geometric lemma that will be the key point in the proof of the $L^p$ estimates on the transport density $\sigma$. So, let the manifold $\mathbb{R}^2$ be equipped with the conformal metric $d_k$. Fix $x_0\in \partial\Omega$ and let $\alpha(s)$, $s \in [-\varepsilon,\varepsilon]$, be an arc in $\partial\Omega$. For each $s \in [-\varepsilon,\varepsilon]$, we denote by $\gamma_s$ the geodesic between $\alpha(s)$ and $x_0$. Let $\nu(s)$ be the initial unit tangent vector to $\gamma_s$ and $\tau(s)$ be the (weighted) length of $\gamma_s$.

\begin{lemma}\label{geometric lemma}
Assume that $\Omega$ is a bounded and geodesically convex domain, and $k \in C^{1,1}(\overline{\Omega})$ with $0<k_{\min} \leq k\leq k_{\max}<\infty$. Then, there exists a constant $C< \infty$ depending only on $\mbox{diam}(\Omega),\,k_{\min},\,k_{\max}$, $||\nabla k||_\infty$ and \,$||D^2 k ||_\infty$
such that, for a.e. $s \in \partial\Omega$, we have
\begin{equation}\label{lower bound on the Jacobian 1 initial}
\det(D_{(s,t)}\gamma_{s}(t)) \geq (1-t)^C\,k^{-1}(s)\,\tau(s)[\nu(s) \cdot n(s)],\,\,\mbox{for all}\,\,\,t \in [0,1].
\end{equation}
\end{lemma}

\begin{proof}
For $s \in [-\varepsilon,\varepsilon]$, let $(e_1,e_2)$ be a basis such that $e_1=\alpha^\prime(s)$ and $e_2= \nu(s)$. Let us parallel-transport along
the geodesic $\gamma_s$ to define a new family of basis $(e_1(t),e_2(t))$. Let [·, ·] be the Lie bracket and $\nabla$ be the Levi-Civita connection. Let us denote by $\exp$ the Riemannian exponential map.
  For all $s\in [-\varepsilon,\varepsilon]$ and $t \in [0,1]$, set 
$$\Psi(s,t):=\gamma_s(t)=\exp_{\alpha(s)}[t\,\tau(s)\,\nu(s)].$$  
Now, 
we define the vector fields $J_1$ and $J_2$ as follows: 
$$J_1(s,t)=\frac{d}{d \delta}_{|\delta=0}\Psi(s+\delta,t)\,\,\,\,\,\mbox{and}\,\,\,\,\,J_2(s,t)=\frac{d}{d \delta}_{|\delta=0}\Psi(s,t+\delta).$$
Set 
$$J(s,t)=(J_1(s,t),J_2(s,t))=D \Psi(s,t) \,\,\,\,\,\mbox{and}\,\,\,\,\,\,\mathcal{J}(s,t)=\det [J(s,t)].$$
Thanks to \cite[Theorem 11.3]{Vil}, one can show that this Jacobian $\mathcal{J}$ cannot vanish, except possibly at the endpoints of the geodesic $\gamma_{s}$.
So, we have 
$$\mathcal{J}^\prime(s,t)=\mbox{tr}[J^\prime(s,t)\,J(s,t)^{-1}]\,\mathcal{J}(s,t).$$\\
Let us denote by $\mathrm{d}\Psi$ the differential map of $\Psi$. The fact that $[\partial_{e_1},\partial_{e_2}]=0$ implies that $J_1$ and $J_2$ commute, since
$$[J_1,J_2]=[\mathrm{d}\Psi(\partial_{e_1}),\mathrm{d}\Psi(\partial_{e_2})]=\mathrm{d}\Psi[\partial_{e_1},\partial_{e_2}]=0.$$
Then, we have 
$$\nabla_{J_1}J_2=\nabla_{J_2}J_1.$$
Yet, 
$$J_1(s,t)=J_{11}(s,t)e_1(t)+J_{21}(s,t)e_2(t).$$ 
Hence, 
$$\nabla_{J_2}J_1=\nabla_{\gamma_s^\prime}J_1=J_{11}^\prime(s,t)e_1(t)+J_{11}(s,t)e_1^\prime(t)+J_{21}^\prime(s,t)e_2(t)+J_{21}(s,t)e_2^\prime(t).$$
But, $e_2^\prime(t)=\nabla_{\gamma_s^\prime}\,\gamma_s^\prime=0$\, since \,$\gamma_s^\prime$\, is a geodesic. In addition, we have $\nabla_{\gamma_s^\prime}\,e_1(t)=\Gamma_{21}^1 e_1(t) + \Gamma_{21}^2 e_2(t)$, where $\Gamma_{21}^1$ and $\Gamma_{21}^2$ denote the Christoffel symbols. Then, we get  
\begin{equation} \label{eq.6.1}
\nabla_{J_2}J_1=
(J_{11}^\prime(s,t)+J_{11}(s,t)  \Gamma^1_{21})e_1(t)+(J_{21}^\prime(s,t)+J_{11}(s,t) \Gamma^2_{21})e_2(t).
\end{equation} 
On the other hand, we have
$$\nabla_{J_1}J_2=J_{11}(s,t)\nabla_{e_1(t)}J_2+J_{21}(s,t)\nabla_{e_2(t)}J_2.$$\\
Let $A(t)$ be the matrix, in the basis $\{e_1(t),e_2(t)\}$, associated with the endomorphism $X \mapsto
\nabla_X J_2$.
Then, one has 
$$J_{11}(s,t)\nabla_{e_1(t)}J_2+J_{21}(s,t)\nabla_{e_2(t)}J_2$$
$$=J_{11}(s,t)[A_{11}(t)e_1(t) +A_{21}(t)e_2(t)]+J_{21}(s,t)[A_{12}(t)e_1(t) +A_{22}(t)e_2(t)]$$
$$=[J_{11}(s,t)A_{11}(t) + J_{21}(s,t)A_{12}(t)]e_1(t) +[J_{11}(s,t)A_{21}(t)+J_{21}(s,t)A_{22}(t)]e_2(t).$$\\
From \eqref{eq.6.1}, we get
$$(J_{11}^\prime(s,t)+J_{11}(s,t)  \Gamma^1_{21})e_1(t)+(J_{21}^\prime(s,t)+J_{11}(s,t) \Gamma^2_{21})e_2(t)$$
$$=[J_{11}(s,t)A_{11}(t) + J_{21}(s,t)A_{12}(t)]e_1(t) +[J_{11}(s,t)A_{21}(t)+J_{21}(s,t)A_{22}(t)]e_2(t).$$
Hence, 
$$J^\prime =[A-B] J,$$
where
$$B=\begin{pmatrix}
\Gamma^1_{21} & 0\\
\Gamma^2_{21} & 0
\end{pmatrix}.$$
This implies that
$$\mathcal{J}^\prime(s,t)=\mbox{tr}[A-B]\,\mathcal{J}(s,t).$$
 A direct computation shows that
$$\Gamma_{21}^1=\partial_{e_2}\log[\sqrt{1-(k(\gamma_s(t))^2\,e_1(t) \cdot e_2(t))^2}]=0,$$ 
 so $\mbox{tr}[B]=0$. Moreover, we have $J_2(s,t)=\gamma_s^\prime(t)=-k^{-1}(\gamma_s(t))\,\tau(s)\,\frac{\nabla d_k(\gamma_s(t),x_0)}{|\nabla d_k(\gamma_s(t),x_0)|}$. Now, it is well known that the distance function $d_k(\cdot,x_0)$ to $x_0$ is locally semiconcave in $\mathbb{R}^2 \backslash \{x_0\}$ with $D^2[d_k(x,x_0)] \leq \frac{C}{d_k(x,x_0)}I$ for some constant $C$ depending on $\mbox{diam}(\Omega)$, $k_{\min}$, $k_{\max}$, $||\nabla k||_\infty$ and $||D^2 k||_\infty$ (see, for instance, \cite{CanSin}). This yields that 
$$\mbox{tr}[A] \geq \frac{-C}{(1-t)}.$$
Hence, 
$$\mathcal{J}^\prime(s,t)\geq \frac{-C}{(1-t)}\,\mathcal{J}(s,t).$$
We infer that
$$\log[\mathcal{J}(s,t)] - \log[\mathcal{J}(s,0)] \geq C\,\log(1-t).$$
Finally, we get
$$\mathcal{J}(s,t)  \geq \mathcal{J}(s,0)\,(1-t)^{C}.$$\\
Yet, $J_1(s,0)=\alpha^\prime(s)$\, and \,$J_2(s,0)=k^{-1}(s)\,\tau(s)\,\nu(s)$. Hence, $\mathcal{J}(s,0)=k^{-1}(s)\,\tau(s)[\nu(s) \cdot n(s)]$. Consequently, we get the following estimate: 
$$\mathcal{J}(s,t)  \geq (1-t)^C\,k^{-1}(s)\,\tau(s)[\nu(s) \cdot n(s)]. \qedhere $$
\end{proof}

In order to prove $L^p$ summability on the transport density $\sigma$, we need also to introduce a definition that generalizes the notion of uniform convexity in the Riemannian case. Let us assume that for all $x,\,y \in \overline{\Omega}$, there is a unique geodesic $\gamma_{x,y}$ from $x$ to $y$. For $x \in \partial\Omega$, let us denote by $\gamma_x$ a geodesic starting at $x$ with initial unit tangent vector $\nu(x)$ such that the endpoint $\gamma_x(1) \in \partial\Omega$ and  by $\tau(x)$ the weighted length of this geodesic $\gamma_x$. It is not difficult to see that a domain $\Omega$ is geodesically convex (resp.  geodesically strictly convex) if and only if $\nu(x) \cdot n(x) \geq0$ (resp. $\nu(x) \cdot n(x) > 0$), for all $x\in\partial\Omega$ and all geodesics $\gamma_x$, where $-n(x)$ is any unit vector in the exterior normal cone to $\Omega$ at $x$.  
Now, we define the uniform geodesic convexity as follows:
 \begin{definition} \label{uniform geodesic convexity definition}
 We say that $\Omega$ is  geodesically uniformly convex if there is a constant $c>0$ such that, for any $x \in \partial\Omega$ and all geodesics $\gamma_x$, $\nu(x) \cdot n(x) \geq c\, \tau(x)$. 
 \end{definition}
 
 Our main result in this section is the following theorem.

\begin{theorem} \label{L^p estimates}
Suppose that $\Omega$ is strictly geodesically convex. Then, if $f^+ \in L^1(\partial\Omega)$, the transport density $\sigma$ belongs to $L^1(\Omega)$. Moreover, for all $p \in [1,2]$, 
we have $\sigma \in L^p(\Omega)$
as soon as $f^+$ and $f^-$ are both in $L^p(\partial\Omega)$ and $\Omega$ is uniformly geodesically convex.  
\end{theorem}

\begin{proof}
We will use a similar strategy as in the proof of \cite[Proposition 3.1]{DS}. First, assume that the target measure $f^-$ is atomic with $n$ atoms $\{x_{i}\,:\,1 \leq i \leq n\}$. Let $\Gamma_i$ be the set of points $x$ on $\mbox{spt}(f^+)$ which are transported to the atom $x_i$ and, let us denote by $\Omega_{i}$ the set of points on the geodesics between $\Gamma_i$ and $x_{i}$. As \,$\Omega$ is geodesically convex, then all these sets $\Omega_{i}$ are subsets of $\overline{\Omega}$. Moreover, thanks to the fact that the transport rays cannot intersect at an interior point of either of them, the sets $\Omega_i$ are essentially disjoint. Let us decompose the transport density $\sigma$ into two parts $\sigma=\sigma^+ + \sigma^-$, where $\sigma^+$ and $\sigma^-$ are defined as follows: for a fixed $\tau \in [0,1]$, we set
$$<\sigma^+,\phi>:=\int_{\overline{\Omega} \times \overline{\Omega}} \int_0^{\tau} \phi(\gamma_{x,y}(t))\,d_k(x,y)\,\mathrm{d}t\,\mathrm{d}\Lambda (x,y),\,\,\,\,\mbox{for all}\,\,\,\,\phi \in C(\overline{\Omega}),$$
and 
$$<\sigma^-,\phi>:=\int_{\overline{\Omega} \times \overline{\Omega}} \int_{\tau}^1 \phi(\gamma_{x,y}(t))\,d_k(x,y)\,\mathrm{d}t\,\mathrm{d}\Lambda (x,y),\,\,\,\,\mbox{for all}\,\,\,\,\phi \in C(\overline{\Omega}).$$ \\  
First, we will show $L^p$ estimates on $\sigma^+$. Recalling Proposition \ref{prop:transportplan}, there is a unique optimal transport map $T$ from $f^+$ to $f^-$. Then, we have
$$<\sigma^+,\phi>=\int_{\partial\Omega} \int_0^{\tau} \phi(\gamma_{x,T(x)}(t))\,d_k(x,T(x))\,\mathrm{d}t\,\mathrm{d}f^+ (x),\,\,\,\,\mbox{for all}\,\,\,\,\phi \in C(\overline{\Omega}).$$ \\ 
For every $x \in \Gamma_{i}$, we have $T(x)=x_{i}$. Hence, we find that
$$<\sigma^+,\phi>=\sum_{i=1}^n\,\int_{\Gamma_{i}} \int_0^{\tau} \phi(\gamma_{x,x_i}(t))\,d_k(x,x_i)\,\mathrm{d}t\,\mathrm{d}f^+ (x),\,\,\,\,\mbox{for all}\,\,\,\,\phi \in C(\overline{\Omega}).$$
Fix $i \in \{1,...,n\}$ and consider $\sigma^+_{i}$ the restriction of $\sigma^+$ to $\Omega_i$. Let $\alpha(s)$, where $s \in [-\varepsilon,\varepsilon]$, be a parameterization of $\Gamma_i$. Then, we have
$$<\sigma_i^+,\phi>=\int_{-\varepsilon}^\varepsilon \int_0^{\tau} \phi(\gamma_{s}(t))\,d_k(\alpha(s),x_i)\,f^+(\alpha(s))\,|\alpha^\prime(s)|\,\mathrm{d}t\,\mathrm{d}s,\,\,\,\,\mbox{for all}\,\,\,\,\phi \in C(\overline{\Omega}),$$
where $\gamma_s$ denotes the geodesic from the point \,$\alpha(s)$\, to the atom \,$x_i$. Set $y=\gamma_s(t)$, for every $s \in [-\varepsilon,\varepsilon]$ and $t \in [0,\tau]$. Then, one has     
$$<\sigma_{i}^+,\phi>= \int_{\Omega_{i}}\phi(y)\, d_k(\alpha(s),x_i)\,f^+(\alpha(s))\,|\alpha^\prime(s)|\,{\mathcal{J}(s,t)}^{-1}\,\mathrm{d}y,\,\,\,\mbox{for all}\,\,\,\,\phi \in C(\overline{\Omega}),$$
where
$$\mathcal{J}(s,t):=\det[D_{(s,t)} \gamma_{s}(t)].$$
Hence, 
$$\sigma_i^+(y)=\frac{|\alpha^\prime(s)|\,d_k(\alpha(s),x_i)\,f^+(\alpha(s))}{\mathcal{J}(s,t)},\,\,\,\,\mbox{for a.e.}\,\,\,y \in \Omega_i.$$ \\
Set $\tau(s)=d_k(\alpha(s),x_i)$.
Then,     
 $$||\sigma_{i}^+||_{L^p(\Omega_{i})}^p$$
\begin{equation} \label{principal estimate}
= \int_{\Omega_{i}}\frac{|\alpha^\prime(s)|^p\tau(s)^p\,{f^+(\alpha(s))}^p}{{\mathcal{J}(s,t)}^p}\,\mathrm{d}y=\int_{-\varepsilon}^\varepsilon\int_{0}^{\tau}\frac{|\alpha^\prime(s)|^p\tau(s)^p\,{f^+(\alpha(s))}^p}{{\mathcal{J}(s,t)}^{p-1}}\,\mathrm{d}t\,\mathrm{d}s. \end{equation} \\
 For $p = 1$, this gives us the following estimate:
$$||\sigma_i^+||_{L^1(\Omega_i)}\leq C||f^+||_{L^1(\Gamma_i)}.$$
Hence, after taking a sum with respect to $i$ and ${{\tau=1}}$, we infer that $\sigma=\sigma^+$ is in $L^1(\Omega)$ and we have   
$$||\sigma||_{L^1(\Omega)}\leq C||f^+||_{L^1(\partial\Omega)}.$$\\
Now, let $\nu(s)$ be the initial tangent vector to $\gamma_s$. 
From Lemma \ref{geometric lemma}, there is a constant $C<\infty$ such that the following estimate holds 
\begin{equation} \label{lower bound on the Jacobian}
\mathcal{J}(s,t) \geq C^{-1}\,(1-t)^C\tau(s)\, [\nu(s) \cdot n(s)].
\end{equation}
If $\Omega$ is additionally uniformly geodesically convex (see Definition \ref{uniform geodesic convexity definition}), then there is a constant $c>0$ such that $\nu(s) \cdot n(s) \geq c \tau(s)$. So, we get that
$$\mathcal{J}(s,t) \geq C^{-1}\,{{(1-t)^C}}\tau(s)^2.$$
Thus, we continue the computation in \eqref{principal estimate} and we choose ${{\tau<1}}$. Then, we get that
$$
||\sigma_{i}^+||_{L^p(\Omega_{i})}^p \leq C^{p-1} {{\int_0^\tau \frac{1}{(1-t)^{C(p-1)}}\,\mathrm{d}t}} \int_{-\varepsilon}^\varepsilon\tau(s)^{2-p}\,|\alpha^\prime(s)|\,{f^+(\alpha(s))}^p\,\mathrm{d}s$$
\begin{equation}\label{eq: 6.50}
\leq C \int_{\Gamma_{i}}d_k(x,x_i)^{2-p}{f^+(x)}^p\,\mathrm{d}\mathcal{H}^1(x).
\end{equation}
For all $p \in [1,2]$, this implies that 
$$||\sigma_i^+||_{L^p(\Omega_i)}\leq C||f^+||_{L^p(\Gamma_i)}.$$
Hence,
$$||\sigma^+||_{L^p(\Omega)}\leq C||f^+||_{L^p(\partial\Omega)}.$$\\
Now, assume that $f^- \in \mathcal{M}^+(\partial\Omega)$ and let $(f_n^-)_n$ be a sequence of atomic measures such that $f_n^- \rightharpoonup f^-$. Let $\sigma_n$ be the transport density between $f^+$ and $f_n^-$. We have
\begin{equation} \label{(7.1)}
  ||\sigma_n^+||_{L^p(\Omega)}\leq C||f^+||_{L^p(\partial\Omega)}.  
\end{equation} 
Thanks to Proposition \ref{stability}, passing to the limit in \eqref{(7.1)} when the number of atoms $n \to \infty$, we infer that the positive measure $\sigma^+$ between $f^+$ and $f^-$ is in $L^p(\Omega)$ as soon as $f^+ \in L^p(\partial\Omega)$. Moreover, $\sigma^+$ satisfies the following estimate:
$$||\sigma^+||_{L^p(\Omega)}\leq C ||f^+||_{L^p(\partial\Omega)}.$$
But now, thanks to the uniqueness of the optimal transport plan $\Lambda$ between $f^+$ and $f^-$, it is not difficult to see that if $f^- \in L^p(\partial\Omega)$, then by approximating the source measure $f^+$ with atomic measures, one can obtain $L^p$ estimates on $\sigma^-$. And, we get that
$$||\sigma^-||_{L^p(\Omega)}\leq C ||f^-||_{L^p(\partial\Omega)}.$$  
Finally, we infer that
$$||\sigma||_{L^p(\Omega)}\leq C ||f||_{L^p(\partial\Omega)},\,\,\mbox{for every}\,\,p \in [1,2]. \qedhere $$ 
\end{proof}

 In addition, under the assumption that the source and target measures $f^\pm$ are smooth enough, one we can prove $L^p$ estimates on the transport density $\sigma$ for large $p$. More precisely, we have the following

\begin{proposition} \label{L^p estimates for large p under smoothness of data}
Suppose that $\Omega$ is uniformly geodesically convex. Assume that $f^\pm \in C^{0,\alpha}(\partial\Omega)$ with $0<\alpha\leq1$. Then, the transport density $\sigma$ between $f^+$ and $f^-$ belongs to $L^p(\Omega)$ with $p=\frac{2}{1-\alpha}$. 
\end{proposition}
\begin{proof}
Let $T^+$ be the optimal transport map from $f^+$ to $f^-$ and set $T^-:=[T^+]^{-1}$. From \eqref{eq: 6.50} and thanks to Proposition \ref{prop:transportplan}, we have the following estimate:
$$||\sigma||_{L^p(\Omega)}^p
\leq C\left[ \int_{\Gamma}d_k(x,T^+(x))^{2-p}{f^+(x)}^p\,\mathrm{d}\mathcal{H}^1(x) + \int_{\Gamma}d_k(T^-(y),y)^{2-p}{f^-(y)}^p\,\mathrm{d}\mathcal{H}^1(y)\right].$$\\
As $f^\pm \in C^{0,\alpha}(\Omega)$, then we get that
$$f^\pm(x)=f^\pm(x) -f^\pm(T^\pm(x)) \leq C|x-T^\pm(x)|^\alpha \leq C d_k(x,T^\pm(x))^\alpha.$$
Hence, 
$$||\sigma||_{L^p(\Omega)}^p
\leq C^p \int_{\Gamma}d_k(x,T^+(x))^{2-p+p\alpha}\,\mathrm{d}\mathcal{H}^1(x)<\infty$$\\
as soon as $p \leq 2/(1-\alpha)$. $\qedhere$
\end{proof}

On the other hand, it is possible to prove $L^p$ estimates on the transport density $\sigma$ for  large $p$ (i.e., $p>2$) if $f^\pm \in L^p(\partial\Omega)$ but under the assumption that $\mbox{spt}(f^+)$ and $\mbox{spt}(f^-)$ do not intersect.

\begin{proposition} \label{L^p estimates for large p}
 Suppose that $\Omega$ is strictly geodesically convex. Let $f^\pm \in L^p(\partial\Omega)$ be such that $\mbox{spt}(f^+) \cap \mbox{spt}(f^-) = \emptyset$. Then, the transport density $\sigma$ belongs to $L^p(\Omega)$, for all $p \in [1,\infty]$. 
\end{proposition}

\begin{proof}
The proof follows directly from the previous estimates in Theorem \ref{L^p estimates}. From \eqref{principal estimate}, we have
$$||\sigma^+||_{L^p(\Omega)}^p\leq \int_{-\varepsilon}^\varepsilon\int_{0}^{\tau}\frac{|\alpha^\prime(s)|^p\,\tau(s)^p\,{f^+(\alpha(s))}^p}{{\mathcal{J}(s,t)}^{p-1}}\,\mathrm{d}t\,\mathrm{d}s.$$\\
From \eqref{lower bound on the Jacobian} and thanks to the strict geodesic convexity of $\Omega$, we infer that there is a uniform constant $C<\infty$ depending only on the geometry of the domain $\Omega$ such that the following estimate holds 
$$\mathcal{J}(s,t) \geq C^{-1}(1-t)^C.$$
For $\tau<1$, this implies that
$$||\sigma^+||_{L^p(\Omega)}\leq C ||f^+||_{L^p(\partial\Omega)}.$$
In the same way, we prove that 
$$||\sigma^-||_{L^p(\Omega)}\leq C ||f^-||_{L^p(\partial\Omega)}.$$ \\
Consequently, we get that the transport density $\sigma$ between $f^+$ and $f^-$ is in $L^p(\Omega)$. Moreover, we have the following estimate:
$$||\sigma||_{L^p(\Omega)}\leq C ||f||_{L^p(\partial\Omega)}. \qedhere$$ 
\end{proof}



\section{Applications to the weighted least gradient problem}\label{sec:applications}

Thanks to the results in the previous sections, we prove existence and uniqueness of a solution $u$ to the weighted least gradient problem. Moreover, we show $W^{1,p}$ regularity on this solution $u$ from the $L^p$ estimates on the transport density $\sigma$. First, we have the following    \begin{theorem}
Assume that $\Omega$ is strictly geodesically convex. Then, for any boundary datum $g \in BV(\partial\Omega)$, the weighted least gradient problem \eqref{eq:weightedleastgradientproblem} admits a solution.
\end{theorem}

\begin{proof}
 Set $f=\partial_\tau g$. Let $\Lambda$ be an optimal transport plan between $f^+$ and $f^-$. Let $\sigma_\Lambda$ be the transport density associated with $\Lambda$ and let $v_\Lambda$ be the vector version of $\sigma_\Lambda$ (see \eqref{Beckmann minimizer}). Then,
$$\sigma_\Lambda(\partial\Omega)=\int_{\partial\Omega \times \partial\Omega} \mathcal{H}^1_k(\gamma_{x,y} \cap \partial\Omega)\,\mathrm{d}\Lambda (x,y),$$ where $\mathcal{H}^1_k(\gamma):=\int_0^1 k(\gamma(t))|\gamma^\prime(t)|\,\mathrm{d}t$. Since $\Omega$ is strictly geodesically convex, for any two points $x,\,y \in \partial\Omega$, the geodesic $\gamma_{x,y}$ lies in the interior of $\Omega$ except for its endpoints, so $\mathcal{H}^1_k(\gamma_{x,y} \cap \partial\Omega)=0$. This implies that $\sigma_\Lambda(\partial\Omega)=0$. From Theorem \ref{thm:anisotropicequivalence}, we infer that there is a function $u \in BV(\Omega)$ such that $v_\Lambda=R_{\frac{\pi}{2}}Du$ and it is a solution for Problem \eqref{eq:weightedleastgradientproblem}.
\end{proof}

\begin{theorem}
Assume that $\Omega$ is strictly geodesically convex and $g \in BV(\partial\Omega) \cap C(\partial\Omega)$. Then, the weighted least gradient problem \eqref{eq:weightedleastgradientproblem} has a unique solution.
\end{theorem}

\begin{proof}
 Since $g \in BV(\partial\Omega) \cap C(\partial\Omega)$, its tangential derivative is a non-atomic measure and so, by Proposition \ref{prop:transportplan}, the problem \eqref{Kantorovich problem} has a unique optimal transport plan $\Lambda$. From the equivalence between Problems \eqref{Kantorovich problem} and \eqref{weighted beckmann problem}, we infer that $v_\Lambda$ is the unique minimizer of Problem \eqref{weighted beckmann problem}. Thanks to Theorem \ref{thm:anisotropicequivalence}, if $u$ is any solution for Problem \eqref{eq:weightedleastgradientproblem} then the measure \,$v:=R_{\frac{\pi}{2}} Du$\, solves Problem \eqref{eq:weightedbeckmannproblem}. This implies that also the solution to Problem \eqref{eq:weightedleastgradientproblem} is unique.
\end{proof}

The main results in this paper are the following $W^{1,p}$ regularity estimates on the solution $u$ of Problem \eqref{eq:weightedleastgradientproblem}:

\begin{theorem}
Suppose that \,$\Omega$ is strictly geodesically convex and $g \in W^{1,1}(\partial\Omega)$. Then, the solution $u$ of Problem \eqref{eq:weightedleastgradientproblem} belongs to $W^{1,1}(\Omega)$. 
\end{theorem}

\begin{proof}
If $g \in W^{1,1}(\partial\Omega)$, then $f=\partial_\tau g \in L^1(\partial\Omega)$. From Theorem \ref{L^p estimates}, we infer that the transport density $\sigma$ between $f^+$ and $f^-$ is in $L^1(\Omega)$. But, thanks to Theorem \ref{thm:anisotropicequivalence}, we have $|Du|=\sigma$, so $u$ belongs to $W^{1,1}(\Omega)$.
\end{proof}

\begin{theorem}
Suppose that $\Omega$ is uniformly geodesically convex and $g \in W^{1,p}(\partial\Omega)$ with $p \in [1,2]$. Then, the solution $u$ of Problem \eqref{eq:weightedleastgradientproblem} is in $W^{1,p}(\Omega)$.  
\end{theorem}

\begin{proof}
If $g \in W^{1,p}(\partial\Omega)$, then $f=\partial_\tau g \in L^p(\partial\Omega)$. From Theorem \ref{L^p estimates}, we infer that the transport density $\sigma$ between $f^+$ and $f^-$ is in $L^p(\Omega)$. So, using Theorem \ref{thm:anisotropicequivalence}, we have $|Du|=\sigma$ and we get $u \in W^{1,p}(\Omega)$.
\end{proof}

\begin{theorem}
Assume that $\Omega$ is uniformly geodesically convex and $g \in C^{1,\alpha}(\partial\Omega)$ with $\alpha \in (0,1]$. Then, the solution $u$ of Problem \eqref{eq:weightedleastgradientproblem} is in $W^{1,p}(\Omega)$ with $p=2/(1-\alpha)$.   
\end{theorem}

\begin{proof}
 If $g \in C^{1,\alpha}(\partial\Omega)$, then $f=\partial_\tau g \in C^{0,\alpha}(\partial\Omega)$. Then, the result follows immediately from Proposition \ref{L^p estimates for large p under smoothness of data}.
\end{proof}

\begin{theorem}
Suppose that $\Omega$ is strictly geodesically convex. Assume that $g \in W^{1,p}(\partial\Omega)$ and $g$ has flat parts separating those with positive and negative derivatives. Then, the solution $u$ of Problem \eqref{eq:weightedleastgradientproblem} belongs to $W^{1,p}(\Omega)$, for all \,$p \in [1,\infty]$. 
\end{theorem}

\begin{proof}
 Notice that $\mbox{spt}(f^+) \cap \mbox{spt}(f^-)=\emptyset$. If $g \in W^{1,p}(\partial\Omega)$, then $f=\partial_\tau g \in L^p(\partial\Omega)$, so the proof follows immediately from Proposition \ref{L^p estimates for large p} and the fact that $|Du|=\sigma$.
\end{proof}

We conclude with a generalisation of \cite[Theorem 4.1]{Gor2021Appl} from the Euclidean case. Let us note that in this case the solution is not necessarily unique, but the result is valid for all the solutions.

\begin{theorem}
Suppose that $\Omega \subset \mathbb{R}^2$ is strictly geodesically convex. Let $g \in SBV(\partial\Omega)$. If $u \in BV(\Omega)$ is a solution to Problem \eqref{eq:leastgradientproblem} with boundary data $g$, then $u \in SBV(\Omega)$.
\end{theorem}

\begin{proof}
The proof of this result in the Euclidean case in \cite{Gor2021Appl} can shortly be described as follows. Since $g \in SBV(\partial\Omega)$, the measure $f = \partial_\tau g$ has no Cantor part, so $f^\pm$ is a union of an absolutely continuous part $f^\pm_{ac}$ and an atomic part $f^\pm_{at}$. We split the transport density into four parts, corresponding to the transport between each of the $f^\pm_{ac}$ and $f^\pm_{at}$.
The transport density between the atomic parts is concentrated on a set of Hausdorff dimension one. The other parts are absolutely continuous, because we need absolute continuity of only one of the source and target measures in order to get $L^1$ estimates for the transport density. Therefore, one can also apply the same reasoning in the case of the weighted least gradient problem; we use the same decomposition of $f^\pm$ into atomic and absolutely continuous parts 
and the fact that in Theorem \ref{L^p estimates} one only needs absolute continuity of the source (or the target) measure.
\end{proof}


{\flushleft \bf Acknowledgements.} WG acknowledges support of the FWF grant I4354, the OeAD-WTZ project CZ 01/2021, and the grant 2017/27/N/ST1/02418 funded by the National Science Centre, Poland.


\begin{bibdiv}
\begin{biblist}

\bib{AB}{article}{
      author={Amar, M.},
      author={Bellettini, G.},
       title={A notion of total variation depending on a metric with
  discontinuous coefficients},
        date={1994},
     journal={Ann. Inst. H. Poincar\'{e} Anal. Non Lin\'{e}aire},
      volume={11},
       pages={91\ndash 133},
}

\bib{AFP}{book}{
      author={Ambrosio, L.},
      author={Fusco, N.},
      author={Pallara, D.},
       title={Functions of bounded variation and free-discontinuity problems},
   publisher={Oxford Math. Monogr.},
     address={Oxford},
        date={2000},
}

\bib{Anz}{article}{
      author={Anzellotti, G.},
       title={Pairings between measures and bounded functions and compensated
  compactness},
        date={1983},
     journal={Ann. di Matematica Pura ed Appl. IV},
      volume={135},
       pages={293\ndash 318},
}
 
\bib{Beckmann}{article}{
author={M. Beckmann}, 
title={A continuous model of transportation}, 
journal={Econometrica}, 
volume={20},
pages={643\ndash660},
date={1952}
 }

\bib{BGG}{article}{
      author={Bombieri, E.},
      author={de Giorgi, E.},
      author={Giusti, E.},
      title={Minimal cones and the {Bernstein} problem},
      journal={Invent. Math.},
      volume={7},
      year={1969},
      pages={243--268}}

\bib{CanSin}{book}{
      author={Cannarsa, P.},
      author={Sinestrari, C.},
       title={Semiconcave Functions, Hamilton—Jacobi Equations, and Optimal Control},
   publisher={Nonlinear Differential Equations and Their Applications, Birkh\"auser},
   address={Basel},
   year={2004}
}

\bib{CF}{article}{
      author={Chen, G-Q.},
      author={Frid, H.},
       title={Divergence-measure fields and hyperbolic conservation laws},
        date={1999},
     journal={Arch. Rational Mech. Anal.},
      volume={147},
       pages={89\ndash 118},
}

\bib{CdC}{article}{
      author={Crasta, G.},
      author={De Cicco, V.},
       title={Anzellotti's pairing theory and the Gauss-Green theorem},
        date={2019},
     journal={Adv. Math.},
      volume={343},
       pages={935 \ndash 970},
}

\bib{DePas1}{article}{
      author={De~Pascale, L.},
      author={Evans, L.C.},
      author={Pratelli, A.},
       title={Integral estimates for transport densities},
        date={2004},
     journal={Bull. of the London Math. Soc.},
      volume={36},
      number={3},
       pages={383\ndash 395},
}

\bib{DePas2}{article}{
      author={Pascale, L.~De},
      author={Pratelli, A.},
       title={Regularity properties for monge transport density and for
  solutions of some shape optimization problem},
        date={2002},
     journal={Calc. Var. Par. Diff. Eq.},
      volume={14},
      number={3},
       pages={249–274},
}

\bib{DePas3}{article}{
      author={Pascale, L.~De},
      author={Pratelli, A.},
       title={Sharp summability for monge transport density via interpolation},
        date={2004},
     journal={ESAIM Control Optim. Calc. Var.},
      volume={10},
      number={4},
       pages={549–552},
}

\bib{DweikWeighted}{article}{
      author={Dweik, S.},
      title={Weighted Beckmann problem with boundary costs},
      date={2018},
      journal={Quarterly of applied mathematics},
      volume={76},
      pages={601--609}
}

\bib{DweikThesis}{book}{
      author={Dweik, S.},
       title={Transport and control problems with boundary costs: regularity and summability of optimal and equilibrium densities},
       publisher={thesis},
       date={2018}
}


\bib{DG2019}{article}{
      author={Dweik, S.},
      author={G\'{o}rny, W.},
       title={Least gradient problem on annuli},
     journal={Analysis \& PDE, to appear},
}

\bib{DS}{article}{
      author={Dweik, S.},
      author={Santambrogio, F.},
       title={{$L^p$} bounds for boundary-to-boundary transport densities, and
  {$W^{1,p}$} bounds for the {BV} least gradient problem in {2D}},
        date={2019},
     journal={Calc. Var. Partial Differential Equations},
      volume={58},
      number={1},
       pages={31},
}

\bib{FmC}{article}{
      author={Feldman, M.},
      author={McCann, R.},
       title={Monge's transport problem on a {R}iemannian manifold},
        date={2002},
     journal={Trans. Amer. Math. Soc.},
      volume={354},
       pages={1667\ndash 1997},
}





\bib{Gor2018CVPDE}{article}{
      author={G\'{o}rny, W.},
       title={Planar least gradient problem: existence, regularity and
  anisotropic case},
        date={2018},
     journal={Calc. Var. Partial Differential Equations},
      volume={57},
      number={4},
       pages={98},
}

\bib{Gor2018JMAA}{article}{
      author={G\'{o}rny, W.},
      title={({Non})uniqueness of minimizers in the least gradient problem},
      journal={J. Math. Anal. Appl.},
      volume={468},
      pages={913--938},
      year={2018}}

\bib{Gor2021IUMJ}{article}{
      author={G\'{o}rny, W.},
      title={Existence of minimisers in the least gradient problem for general boundary data},
      date={2021},
      journal={Indiana Univ. Math. J.},
      volume={70},
      number={3},
      pages={1003--1037},
}


\bib{Gor2021Appl}{article}{
      author={G\'{o}rny, W.},
      title={Applications of optimal transport methods in the least gradient problem},
      date={2021},
      journal={preprint, available at arXiv:$2102.05887$},
      }


\bib{GRS2017NA}{article}{
      author={G\'{o}rny, W.},
      author={Rybka, P.},
      author={Sabra, A.},
       title={Special cases of the planar least gradient problem},
        date={2017},
     journal={Nonlinear Anal.},
      volume={151},
       pages={66\ndash 95},
}

\bib{HM}{article}{
    author={Hauer, D.},
    author={Maz\'{o}n, J.M.},
    title={The {D}irichlet-to-{N}eumann operator associated with the $1$-{L}aplacian and evolution problems},
    journal={Calc. Var. Partial Differential Equations, to appear}
}

\bib{JMN}{article}{
      author={Jerrard, R.L.},
      author={Moradifam, A.},
      author={Nachman, A.I.},
       title={Existence and uniqueness of minimizers of general least gradient
  problems},
        date={2018},
     journal={J. Reine Angew. Math.},
      volume={734},
       pages={71\ndash 97},
}

\bib{Kantorovich}{article}{
author={L. Kantorovich}, 
title={On the transfer of masses}, 
journal={Dokl. Acad. Nauk. USSR}, 
volume={37}, 
pages={7 \ndash8}, 
date={1942}
}

\bib{LMSS}{article}{
      author={Lahti, P.},
      author={Mal\'y, L.},
      author={Shanmugalingam, N.},
      author={Speight, G.},
       title={Domains in metric measure spaces with boundary of positive mean curvature, and the {D}irichlet problem for functions of least gradient},
        date={2019},
     journal={J. Geom. Anal.},
      volume={29},
      number={4},
       pages={3176\ndash 3220},
}

\bib{Mag}{book}{
      author={Maggi, F.},
       title={Sets of finite perimeter and geometric variational problems: an
  introduction to {Geometric} {Measure} {Theory}},
      series={Cambridge Studies in Advanced Mathematics},
   publisher={Cambridge University Press},
     address={Cambridge},
        date={2012},
}


\bib{Maz}{article}{
      author={Maz\'on, J.M.},
      title={The {Euler}-{Lagrange} equation for the anisotropic least gradient problem},
      journal={Nonlinear Anal. Real World Appl.},
      volume={31},
      year={2016},
      pages={452--472}}
      

\bib{MRL}{article}{
      author={Maz\'on, J.M.},
      author={Rossi, J.D.},
      author={Segura de Le\'on, S.}, 
      title={Functions of least gradient and 1-harmonic functions},
      journal={Indiana Univ. Math. J.},
      volume={63},
      year={2014},
      pages={1067--1084}}


\bib{MazPob}{book}{
      author={Maz'ya, V.},
      author={Poborchi, S.},
       title={Differentiable functions on bad domains},
   publisher={World Scientific Publishing Co., Inc.},
     address={River Edge, NJ},
        date={1997},
}


\bib{Monge}{article}{
author={G. Monge}, 
title={M\'emoire sur la th\'eorie des d\'eblais et des remblais}, 
journal={Histoire de l'Acad\'emie Royale des Sciences de Paris}, 
date={1781}, 
pages={666\ndash704}
}


\bib{MNT}{article}{
     author={Moradifam, A.},
     author={Nachman, A.I.},
     author={Tamasan, A.},
    title={Uniqueness of weighted least gradient problems arising in conductivity imaging},
    date={2018},
    journal={Calc. Var. Partial Differential Equations},
    volume={57},
    number={1},
    pages={6}
}






\bib{Mor}{article}{
      author={Moradifam, A.},
       title={Existence and structure of minimizers of least gradient
  problems},
        date={2018},
     journal={Indiana Univ. Math. J.},
      volume={67},
      number={3},
       pages={1025\ndash 1037},
}


\bib{Prat}{article}{
      author={Pratelli, A.},
       title={Equivalence between some definitions for the optimal mass
  transport problem and for the transport density on manifolds},
        date={2005},
     journal={Ann. Mat. Pura Appl.},
      volume={184},
      number={2},
       pages={215\ndash 238},
}

\bib{Sabra}{article}{
author={P. Rybka and A. Sabra}, 
title={The planar Least Gradient problem in convex domains, the case of continuous datum}, 
journal={Nonlinear Anal.},
volume={214},
pages={112595},
date={2022}
}

\bib{San}{article}{
      author={Santambrogio, F.},
       title={Absolute continuity and summability of transport densities:
  simpler proofs and new estimates},
        date={2009},
     journal={Calc. Var. Partial Differential Equations},
      volume={36},
       pages={343\ndash 354},
}

\bib{San2015}{book}{
      author={Santambrogio, F.},
       title={Optimal transport for applied mathematicians},
      series={Progress in Nonlinear Differential Equations and Their
  Applications 87},
   publisher={Birkh\"auser},
     address={Basel},
        date={2015},
}

\bib{Spradlin}{article}{
author={G. Spradlin},
author={A. Tamasan},
title={Not all traces on the circle come from functions of least gradient in the disk}, 
journal={Indiana Univ. Math. J.}, 
volume={63},
date={2014},
pages={1819 \ndash 1837}
}

\bib{Sternberg}{article}{
author={P. Sternberg},
author={G. Williams}, 
author={W.P. Ziemer},
title={Existence, uniqueness, and regularity for functions of least gradient}, 
journal={J. Reine Angew. Math.}, 
volume={430}, 
date={1992},
pages={35 \ndash 60}
 }


\bib{Vil}{book}{
      author={Villani, C.},
       title={Topics in optimal transportation},
   publisher={American Mathematical Society, Graduate Studies in Mathematics Vol. 58},
        date={2003},
}

\bib{Zun}{article}{
      author={Zuniga, A.},
       title={Continuity of minimizers to the weighted least gradient problems},
        date={2019},
     journal={Nonlinear Analysis},
      volume={178},
       pages={86\ndash 109},
}

\end{biblist}
\end{bibdiv}
\end{document}